\documentclass[11pt]{amsart}
\usepackage{tgpagella}
\usepackage{euler}
\usepackage[T1]{fontenc}
\usepackage{amsmath, amssymb}
\allowdisplaybreaks
\usepackage[hidelinks]{hyperref}
\usepackage[english]{babel}
\usepackage{mathrsfs}
\usepackage{tikz-cd}
\usepackage{eucal}

\usepackage[
backend=biber,
style=numeric,
sorting=nyt
]{biblatex}
\usepackage[all]{xy}
\newtheorem{thm}{Theorem}[section]

\newtheorem{defn}[thm]{Definition}

\newtheorem{conv}[thm]{Convention}
\newtheorem{lemma}[thm]{Lemma}
\newtheorem{prop}[thm]{Proposition}
\newtheorem{cor}[thm]{Corollary}

\newtheorem{rmk}[thm]{Remark}
\newtheorem{rmks}[thm]{Remarks}
\newtheorem{example}[thm]{Example}

\makeatletter

\def\indsym#1#2{%
  \setbox0=\hbox{$\m@th#1x$}%
  \kern\wd0%
  \hbox to 0pt{\hss$\m@th#1\mid$\hbox to 0pt{$\m@th#1^{#2}$}\hss}%
  \lower.9\ht0\hbox to 0pt{\hss$\m@th#1\smile$\hss}%
  \kern\wd0}

\def\nindsym#1#2{%
  \setbox0=\hbox{$\m@th#1x$}%
  \kern\wd0%
  \hbox to 0pt{\hss$\m@th#1\not$\kern1.4\wd0\hss}
  \hbox to 0pt{\hss$\m@th#1\mid$\hbox to 0pt{$\m@th#1^{\,#2}$}\hss}%
  \lower.9\ht0\hbox to 0pt{\hss$\m@th#1\smile$\hss}%
  \kern\wd0}

\def\dotminussym#1#2{%
  \setbox0=\hbox{$\m@th#1-$}%
  \kern.5\wd0%
  \hbox to 0pt{\hss\hbox{$\m@th#1-$}\hss}%
  \raise.6\ht0\hbox to 0pt{\hss$\m@th#1.$\hss}%
  \kern.5\wd0}

\def \Proj{\operatorname{proj}}
\def \Span{\operatorname{span}}
\def \Dom{\operatorname{dom}}
\def \St{\operatorname{st}}
\def \Id{\operatorname{id}}
\newcounter{pointnumber}

\title[Nonstandard approach to the direct integral spectral theorem]{A nonstandard approach to the direct integral version of the spectral theorem}
\author{Isaac Goldbring}
\thanks{I. Goldbring was partially supported by NSF grant DMS-2054477. }

\address{Department of Mathematics\\University of California, Irvine, 340 Rowland Hall (Bldg.\# 400),
Irvine, CA 92697-3875}
\email{isaac@math.uci.edu}
\urladdr{http://www.math.uci.edu/~isaac}

\author{Fabrice Nonez}
\address{Department of Mathematics \& Statistics\\ Concordia University, 1455 De Maisonneuve Blvd. W.,
Montréal, Québec, Canada H3G 1M8}
\email{fabrice.nonez@mail.concordia.ca}

\begin{document}

\begin{abstract}
We use nonstandard methods to prove the direct integral version of the Spectral Theorem for Unbounded Self-adjoint Operators.  Our proof avoids the standard reduction to the case of bounded normal operators via the Cayley transform and, as such, works uniformly for both real and complex Hilbert spaces.  Our method also yields a new nonstandard proof of the spectral measure version of the Spectral Theorem.  
\end{abstract}

\maketitle

\section{Introduction}

\subsection{Statement of the main result}  

Given a direct integral $\int_{\mathbb{R}}^\oplus H_td\mu(t)$ of Hilbert spaces\footnote{See the appendix of this paper for more information on direct integrals of Hilbert spaces.}, the ``multiplication operator'' $T$ on $\int_{\mathbb{R}}^\oplus H_td\mu(t)$ given by $T(X)(t):=t\cdot X(t)$ is an unbounded self-adjoint operator\footnote{We assume the reader is familiar with the basics of unbounded operators on Hilbert spaces.  The reader unfamiliar with this material can consult \cite[Chapter X]{conway2019course}.}.  One version of the Spectral Theorem for Unbounded Operators states that all unbounded self-adjoint operators on a separable Hilbert space are essentially of this form.  More precisely, if $A$ is an unbounded self-adjoint operator on the Hilbert space $H$, then there is a unitary $U:H\to \int_{\mathbb{R}}^\oplus H_td\mu(t)$ for some direct integral $\int_{\mathbb{R}}^\oplus H_td\mu(t)$ such that $U(A(x))=T(U(x))$ for all $x\in \Dom(A)$.  Viewing the direct integral construction as the natural generalization of the direct sum construction, this version of the Spectral Theorem is the obvious generalization of the usual finite-dimensional Spectral Theorem in terms of diagonalizable operators.

In this paper, we use the nonstandard methods inspired by \cite{GOLDBRING2021590} and \cite{nonez2024spectralequivalencesnonstandardsamplings} to prove the previous theorem.  We actually prove the above theorem in a slightly more general form:

\begin{thm}\label{theorem_spectral_direct}
	Given a (real or complex) separable Hilbert space $H$ and an unbounded symmetric operator\footnote{We always assume that our unbounded operators are densely defined.} $A$ on $H$, there exists a Borel probability measure $\mu$ on $\mathbb{R}$, a measurable family $\{H_{t} \}_{t\in\mathbb{R}}$ of Hilbert spaces, and an isometry $$U:H\rightarrow\int_{\mathbb{R}}^{\oplus}H_{t}d\mu(t)$$ such that, for any $x\in\Dom(A)$, $U(Ax)(t)=t\cdot \left(U(x)(t)\right)$ for $\mu$-almost all $
t\in\mathbb{R}$.
\end{thm}

As a byproduct of our analysis, in Section 6, we will see that if $A$ is self-adjoint, then the isometry $U$ that we construct will be surjective (and thus unitary).

\subsection{Overview of the proof}

Unfortunately, the quality of a nonstandard proof is often judged by how much shorter it is than the standard argument.  The reader of this paper might look at its length and judge that the nonstandard proof is longer than its standard counterpart!  In our opinion, the correct yardstick for measuring the quality of a nonstandard argument is the clarity of the structure of the proof and of its pedagogical advantage over the standard argument.  In the case of the current paper, we believe that this is certainly true for the proof of the direct integral spectral theorem being presented here.  In this subsection, we offer a high-level overview of the proof.  The rest of the paper can be viewed as simply checking that the details and calculations are as expected.

As with many nonstandard arguments in functional analysis, the basic idea is to ``model'' the situation by a hyperfinite-dimensional subspace of the nonstandard extension.  In our case, we should consider a hyperfinite-dimensional subspace $\tilde{H}$ of ${}^*H$ and an internal symmetric (equiv. self-adjoint) operator $\tilde{A}:\tilde{H}\to \tilde{H}$ which sufficiently approximates $A$ in a certain sense.  While our proof is carried out under much more general assumptions (see the definitions in the next section), the most natural thing to do is to start with an orthonormal basis $(e_j)_{j\in \mathbb{N}}$ for $H$ consisting of vectors from $\Dom(A)$, extend it to a hyperfinite sequence $(e_j)_{j=1}^N$ from ${}^*H$, let $\tilde{H}$ denote the internal span of this hyperfinite sequence, and let $\tilde{A}$ be $\Proj_{\tilde{H}}\circ A|_{\tilde{H}}$, the composition of the restriction of (the nonstandard extension of) $A$ to $\tilde{H}$ with the (internal) orthogonal projection operator onto $\tilde{H}$.

The main idea behind the proof is to apply the (transferred version of the) finite-dimensional version of the spectral theorem to $(\tilde{H},\tilde{A})$.  As such, we have that $\tilde{H}$ is the (internal) direct sum of its eigenspaces $\tilde{H}_\lambda$ corresponding to the eigenvalues $\lambda \in {}^*\mathbb{R}$.  One may consider any weights $c_j\in {}^*\mathbb{R}_{\geq 0}$ for $j\in [N]$ and define an internal measure on the set $\tilde{\sigma}$ of eigenvalues by declaring $\tilde{\mu}(\{\lambda\}):=\sum_{j=1}^N c_j\|\Proj_{H_\lambda}e_j\|^2$; for this measure to be a probability measure, we further require $\sum_{j=1}^N c_j\|e_j\|^2=1$.  The map $$x\mapsto (\lambda\mapsto \frac{1}{\sqrt{\tilde{\mu}}(\lambda)}\Proj_{\tilde{H}_\lambda}x)$$ is an internal unitary map between $\tilde{H}$ and the internal direct integral $\int^\oplus_{\tilde{\sigma}} \tilde{H}_\lambda d\tilde{\mu}(\lambda)$ which intertwines $\tilde{A}$ and the canonical multiplication operator on the internal direct integral.

In order to obtain the statement of Theorem \ref{theorem_spectral_direct}, one somehow needs to take the ``standard part'' of the picture from the previous paragraph.  One immediate observation is that, for one to not lose too much information from such a procedure, one would want that almost all (with respect to the Loeb measure $\mu_L$ induced by $\tilde{\mu}$) eigenvalues are finite.  This can be achieved with one further condition on the weights $(c_j)_{j\in [N]}$, namely that $\sum_{j\in \mathbb{N}}\St(c_j)=1$.  One can then pushforward $\mu_L$ along the standard part map to obtain a Borel probability measure $\mu$ on $\mathbb{R}$.

In what sense should one take the ``standard part'' of the internal direct integral?  The mindset taken in this paper is that the information carried by the integral direct integral is contained in the inner product function $\lambda\mapsto \langle\tilde{U}(x)(\lambda),\tilde{U}(y)(\lambda)\rangle$, where, for $x\in H$, one defines $\tilde{U}(x)\in \prod_\lambda \tilde{H}_\lambda$ by $\tilde{U}(x)(\lambda):=\frac{1}{\sqrt{\tilde{\mu}}(\lambda)}\Proj_{\tilde{H}_\lambda}x$.  An important technical result (Theorem \ref{theorem_nearstandard_sintegrable} below) is that this function is an \emph{S-integrable function}.  As a result, we are entitled to consider the $\mu_L$-integrable function $f^{x,y}(\lambda):=\St \langle \tilde{U}(x)(\lambda),\tilde{U}(y)(\lambda)\rangle$.  We can use $f^{x,y}$ to define a measure $f^{x,y}d\mu_L$ on $\tilde{\sigma}$, which we can then push forward along the standard part map to a Borel probability measure $\nu^{x,y}$ on $\mathbb{R}$.  It is in this sense that we take the standard part of the internal direct integral.

So how does one obtain the conclusion of Theorem \ref{theorem_spectral_direct} from this family of measures?  The idea is to reverse the steps of the previous paragraph.  Indeed, for each $j,l\in \mathbb{N}$, one can consider the Radon-Nikodym derivative $U_{j,l}:=\frac{d\nu^{e_j,e_l}}{d\mu}$ and then show that there is a sequence $(V_j:\mathbb{R}\to \ell_2)_{j\in \mathbb{N}}$ of measurable functions for which $\langle V_j(t),V_l(t)\rangle=U_{j,l}(t)$ for $\mu$-almost all $t\in \mathbb{R}$.  This allows one to conclude that mapping $e_j$ to $V_j$ extends to an isometry $U:H\to \int^{\oplus}_{\mathbb{R}}\ell_2^{\mathbb{K}}d\mu$.  For the purposes of obtaining a unitary operator in case $A$ is self-adjoint, it is better to view $U$ as taking its values in $\int^\oplus_{\mathbb{R}} H_td\mu$, where $H_t$ is the subspace of $\ell_2^{\mathbb{K}}$ spanned by $\{V_j(t) \ : \ j\in \mathbb{N}\}$, and where $(H_t)_{t\in \mathbb{R}}$ is endowed with its ``induced'' measurable structure.\footnote{A short argument is needed to show that this makes sense.}

It is straightforward to show that the standard part of the graph of $\tilde{A}$ is the graph of a closed symmetric operator $\hat A$ extending $A$.  In Section 6, we show that our construction is robust enough so that $\hat A$ is self-adjoint if and only if $U$ is unitary.

One point worth making is that our strategy holds irrespective of whether the Hilbert space is a real Hilbert space or a complex Hilbert space.  The usual approach to the direct integral version of the spectral theorem for unbounded self-adjoint operators uses the Cayley transform to reduce to the case of bounded unitary operators, and such a reduction is only possible when the field of scalars is the complex field, and when the operator is self-adjoint (see, for example, \cite[Chapter X]{conway2019course} or \cite[Chapter 10]{hall2013quantum}). We also note that our proof relies only on the spectral theorem for self-adjoint operators on finite-dimensional Hilbert spaces, which is much easier to prove than any of its infinite-dimensional counterparts.  

The methods use in this paper are partly inspired by those used in \cite{nonez2024spectralequivalencesnonstandardsamplings}, and share apparent similarities, notably when it comes to samplings and scales. In that body of work, the flexibility afforded by these objects are further used to work with specific examples of symmetric operators, such as $-i\frac{d}{dx}$ on $C_c^{\infty}(\mathbb{R})$, to establish explicit isometries (in this case, the resulting isometry being the Fourier transform). We think such applications might be possible in our context as well, helping one find clear descriptions of the measurable families for a given symmetric operator. In that case, such descriptions would most likely provide a more complete picture of "eigenvalue" multiplicity.  We leave this possibility for a future project. 

Finally, we mention that our methods quite easily yield yet another nonstandard proof of the projection-valued measure version of the Spectral Theorem, first established in \cite{GOLDBRING2021590} and then again in \cite{nonez2024spectralequivalencesnonstandardsamplings} (see also \cite{matsunaga2024shortnonstandardproofspectral}).  This derivation is given in Section 7.

\subsection{Conventions}

Throughout this paper, we fix a separable Hilbert space $H$ over the field $\mathbb{K}\in \{\mathbb{R},\mathbb{C}\}$ and a symmetric unbounded operator $A$ on $H$, which we assume is densely defined. The inner product of $x,y\in H$, which is linear in $x$ (thus conjugate linear in $y$), will be denoted with $\langle x,y\rangle.$ Meanwhile, the ordered pair from $H\times H$ will be denoted $(x,y).$ 

If $\{H_x\}_{x\in M}$ is a family of $\mathbb{K}$-Hilbert spaces, then given $X,Y\in\prod_{x\in M}H_x,$ we let $\langle X,Y\rangle$ denote the map $x\rightarrow \langle X(x), Y(x)\rangle:M\rightarrow\mathbb{K}.$ In contrast, if $\mu$ is a measure on some $\sigma$-algebra on $M$ for which the function $\langle X, Y\rangle$ is $\mu$-integrable, then we set $\langle X,Y\rangle_{\mu}=\int_M\langle X, Y\rangle d\mu.$ Similarly, $\|X\|$ is the map $x\rightarrow \|X(x)\|,$ while $\|X\|_{\mu}=\sqrt{\langle X,X\rangle_{\mu}}=\sqrt{\int_M \|X\|^2d\mu}$ when applicable.

We assume that our reader is familiar with the basics of nonstandard analysis.  The reader unfamiliar with these methods may consult \cite{nsarecent}.  The only somewhat nonstandard nonstandard ingredient\footnote{Pun intended.} in our work is the notion of an S-integrable function; we include the basic information about this notion when it is first needed in Section 4.  

As is customary, we assume our nonstandard extension is $\aleph_1$-saturated; this suffices since our Hilbert space is separable.

In order to prevent the paper from becoming cluttered with too many stars, we omit stars on the nonstandard extensions of functions.  Similarly, when considering an internal function on an internal measure space, we write $\int fd\mu$ rather than ${}^*\int fd\mu$ for the internal integral.  Another instance of this convention is that, for an internal subspace $\tilde{H}$ of ${}^*H$, we write $\Proj_{\tilde{H}}$ for the internal operator on ${}^*H$ which projects onto $\tilde{H}$.

Finally, we assume that the natural numbers $\mathbb{N}$ begin at $1$.  When referring to the set of nonnegative integers, we write $\mathbb{Z}_{\geq 0}$.
\section{Quasi-samplings}
Our definitions in this section are heavily inspired by the notion of sampling introduced in \cite{nonez2024spectralequivalencesnonstandardsamplings}. However, not all of the structure of samplings is needed for our current purposes and  so we define the following weaker notion:
\begin{defn}\label{definition_sampling}
	The pair $(\tilde{H}, \tilde{A})$ is called a \textbf{quasi-sampling} for $A$ if:
	\begin{enumerate}
		\item\label{sampling_property_hyperfinite_space} $\tilde{H}$ is a hyperfinite-dimensional subspace of ${}^*H$; 
		\item\label{sampling_property_symmetric_map} $\tilde{A}:\tilde{H}\rightarrow\tilde{H}$ is an internal symmetric operator;
		\item\label{sampling_property_approximation} We have the graph relation $G(A)\subseteq\St(G(\tilde{A})).$ In other words, for each $x\in\Dom(A)$, there exists $\bar{x}\in\tilde{H}$ such that $x=\St(\bar{x})$ and $Ax=\St(\tilde{A}\bar{x}).$ 
		\setcounter{pointnumber}{\value{enumi}}
	\end{enumerate}
\end{defn}

We next recall the definition of a standard-biased scale from \cite[Definition 2.4]{nonez2024spectralequivalencesnonstandardsamplings}.  We also take the opportunity to say what it means for such a scale to be compatible with a quasi-sampling.
\begin{defn}\label{definition_scale}
	If $(\tilde{e}_j)_{j=1}^N$ is a hyperfinite sequence from ${}^*H$, and $(\tilde{c}_j)_{j=1}^N$ is a hyperfinite sequence from ${}^*\mathbb{R}_{>0}$, then the pair $\left((\tilde{e}_j)_{j=1}^N,(\tilde{c}_j)_{j=1}^N\right)$ is called a \textbf{standard-biased scale} if:
	\begin{enumerate}
		\setcounter{enumi}{\value{pointnumber}}
		\item\label{scale_property_nearstandard} for each $j\in\mathbb{N}$, $\tilde{e}_j$ is nearstandard in $H$, $\tilde{c}_j$ is near standard in $\mathbb{R}$, and $\St(\tilde{c}_j\tilde{e}_j)\neq 0$.
		\item\label{scale_property_density} the set $\{\St(\tilde{e}_j)\}_{j\in\mathbb{N}}$ is dense-spanning in $H$.
		\item\label{scale_property_probability} $\sum_{j=1}^{N}\tilde{c}_j\|\tilde{e}_j\|^2=1.$
		\item \label{scale_property_standard_bias}$\sum_{j\in\mathbb{N}}\St(\tilde{c}_j)\|\St(\tilde{e}_j )\|^2=1.$
		\setcounter{pointnumber}{\value{enumi}}
	\end{enumerate}
	Furthermore, we say that this scale is compatible with the quasi-sampling $(\tilde{H},\tilde{A})$ if:
	\begin{enumerate}
		\setcounter{enumi}{\value{pointnumber}}
		\item\label{scale_compat_property_H} for each $j\in[N]$, $\tilde{e}_j\in\tilde{H}$.
		\item\label{scale_compat_property_A} for each $j\in\mathbb{N}$, $\|\tilde{A}\tilde{e}_j\|$ is finite. 
	\end{enumerate}
\end{defn}

The results of \cite[Section 2]{nonez2024spectralequivalencesnonstandardsamplings} show that quasi-samplings for $A$ and standard-biased scales compatible with a given quasi-sampling exist.  We provide one such construction here, constructing the quasi-sampling and scale at the same time; the details are left to the reader (or they may consult \cite[Section 2]{nonez2024spectralequivalencesnonstandardsamplings}). 

Let $(e_j)_{j\in \mathbb{N}}$ be an orthonormal subset of $\Dom(A)$ such that $\{(e_j, Ae_j)\}_{j\in\mathbb{N}}$ is dense-spanning in $G(A)$.  Such a set must necessarily be an orthonormal basis of $H$. Then, for some infinite hypernatural $N,$ let $\tilde{H}$ be the internal subspace of ${}^*H$ spanned by $({}^*e_j)_{j=1}^N$ and set $\tilde{A}=(\Proj_{\tilde{H}}\circ {}^*A)|_{\tilde{H}}.$ Finally, for $j\in [N]$, set $\tilde{e}_j={}^*e_j$ and $\tilde{c}_j:=\frac{1}{2^j(1-2^{-N})}.$ Then it is straightforward to verify that $(\tilde{H}, \tilde{A})$ is a quasi-sampling for which $((\tilde{e}_j)_{j=1}^N,(\tilde{c}_j)_{j=1}^N)$ is a compatible standard-biased scale.




\begin{conv}
    	For the remainder of this paper, we fix a quasi-sampling $(\tilde{H},\tilde{A})$ for $A$ as well as a standard-biased scale $\left((\tilde{e}_j)_{j=1}^N,(\tilde{c}_j)_{j=1}^N\right)$ compatible with $(\tilde{H},\tilde{A})$. Furthermore, for $j\in\mathbb{N}$, we set $c_j:=\St(\tilde{c}_j)$ and $e_j:=\St(\tilde{e}_j).$
\end{conv}

\section{The induced measure}
We begin with a few preliminary definitions.
\begin{defn}\label{definition_internal_eigenspaces}
	For each $\lambda\in {}^*\mathbb{R}$, we set $\tilde{H}_{\lambda}=\ker(\lambda I_{\tilde{H}}-\tilde{A})$, while for internal $V\subseteq {}^*\mathbb{R},$ we set $\tilde{H}_V={}^*\bigoplus_{\lambda\in V} \tilde{H}_{\lambda}$ (so $\tilde{H}_\lambda=\tilde{H}_{\{\lambda\}}$). 
\end{defn}

\begin{rmks}\label{remark_internal_eigenspaces}

\

\begin{enumerate}
    \item   Since $\tilde{A}$ is symmetric, we have that all $\tilde{H}_{\lambda}$ are orthogonal to each other. 
    \item Since $\tilde{A}$ is symmetric, we have $\Proj_{\tilde{H}_{\lambda}}\tilde{A}x=\tilde{A}\Proj_{\tilde{H}_{\lambda}}x$ for all $x\in\tilde{H}$.
    \item Since $\tilde{H}$ is hyperfinite-dimensional, we have $\{\lambda\in{}^*\mathbb{R}\;|\; \tilde{H}_{\lambda}\neq (0) \}$ is hyperfinite. 
    \item By the Transfer Principle and the finite-dimensional version of the Spectral Theorem, we have $\tilde{H}=\tilde{H}_{{}^*\mathbb{R}}$.
    \end{enumerate}
\end{rmks}

\begin{defn}
	Let $\tilde{\sigma}= \{\lambda\in{}^*\mathbb{R}\;|\; \{\tilde{e}_j\}_{j=1}^N\not\subseteq \tilde{H}_{\lambda}^{\perp} \}$. In other words, $\lambda\in\tilde{\sigma}$ if and only if there exists $j\in [N]$ for which ${}^*\Proj_{\tilde{H}_{\lambda}}\tilde{e}_j\neq 0$.
\end{defn}
\begin{rmk}\label{remark_sigma_space_density}
	Since $\lambda\in\tilde{\sigma}$ implies $\tilde{H}_{\lambda}\neq(0)$, we have that $\tilde{\sigma}$ is hyperfinite. 
\end{rmk}
\begin{prop}\label{proposition_internal_spectrum_compat_scale}
	For each $j\in[N]$, we have $\tilde{e}_j\in\tilde{H}_{\tilde{\sigma}}.$
\end{prop}	
\begin{proof}
By Remark \ref{remark_internal_eigenspaces}, we have that $\tilde{H}=\tilde{H}_{\tilde{\sigma}}\oplus\tilde{H}_{{}^*\mathbb{R}\setminus\tilde{\sigma}}$. Since Property \ref{scale_compat_property_H} of Definition \ref{definition_scale} establishes that $\tilde{e}_j\in\tilde{H}$, we only need to show that $\tilde{e}_j\perp\tilde{H}_{{}^*\mathbb{R}\setminus\tilde{\sigma}} .$ However, since the sum in Definition \ref{definition_internal_eigenspaces} is (internally) algebraic, and since by definition $\tilde{e}_j\in \tilde{H}_{\lambda}^\perp$ for each $\lambda\in{}^*\mathbb{R}\setminus\tilde{\sigma},$ we indeed have $\tilde{e}_j\perp\tilde{H}_{{}^*\mathbb{R}\setminus\tilde{\sigma}} .$
\end{proof}

\begin{rmk}
In the construction of a quasi-sampling and standard-biased scale given in the previous section, we have that $\tilde{\sigma}$ is simply the collection of eigenvalues of $\tilde{A}$ and that $\tilde{H}_{\tilde{\sigma}}=\tilde{H}$.  The reader will lose nothing moving forward if they keep this special case in mind. 
\end{rmk}

We now work towards defining our main measure.

\begin{defn}
	Let $\tilde{\mathcal{A}}={}^*\mathcal{P}(\tilde{\sigma})$ be the internal algebra of internal subsets of $\tilde{\sigma}$ and define $\tilde{\mu}:\tilde{\mathcal{A}}\rightarrow{}^*\mathbb{R}_{\geq0}$ by
	$$\tilde{\mu}(V):=\sum_{j=1}^N \tilde{c}_j\|{}^*\Proj_{\tilde{H}_V}\tilde{e}_j\|^2. $$
For $\lambda\in\tilde{\sigma}$, we set $\tilde{\mu}(\lambda):=\tilde{\mu}(\{\lambda\})=\sum_{j=1}^N\tilde{c}_j\|{}^*\Proj_{\tilde{H}_{\lambda}}\tilde{e}_j\|^2.$
\end{defn}
\begin{rmk}
	By the definition of $\tilde{\sigma}$, we have that $\tilde{\mu}(\lambda)>0$ for every $\lambda\in\tilde{\sigma}$. 
\end{rmk}

By Remark \ref{remark_internal_eigenspaces}, Proposition \ref{proposition_internal_spectrum_compat_scale}, and Property \ref{scale_property_probability} of Definition \ref{definition_scale}, we directly conclude the following:

\begin{prop}
	$(\tilde{\sigma},\tilde{\mathcal{A}},\tilde{\mu})$ is an internal probability space.
\end{prop}

\begin{defn}
	Let $(\tilde{\sigma},\mathcal{A}_L,\mu_L)$ be the Loeb probability space associated to $(\tilde{\sigma},\tilde{\mathcal{A}},\tilde{\mu})$. 
\end{defn}

The goal is to push forward $\mu_L$ via the standard part map to a measure on $\mathbb{R}$. In order to do so, we need the following proposition:

\begin{prop}\label{proposition_almostall_eigenvalues_finite}
	$\operatorname{Fin}({}^*\mathbb{R})\cap\tilde{\sigma}\in\mathcal{A}_L$ and $\mu_L(\tilde{\sigma}\setminus\operatorname{Fin}({}^*\mathbb{R}) )=0$.
\end{prop}
\begin{proof}
	Set $\operatorname{Fin}(\tilde{\sigma})=\operatorname{Fin}({}^*\mathbb{R})\cap\tilde{\sigma}.$ We have $\operatorname{Fin}(\tilde{\sigma})=\bigcup_{n\in\mathbb{N}}(\tilde{\sigma}\cap{}^*[-n,n])$. Since all the sets in the union are in $\tilde{\mathcal{A}}\subset\mathcal{A}_L$, we have that  $\operatorname{Fin}(\tilde{\sigma})\in\mathcal{A}_L$, proving the first statement of the proposition. Note also that $$\mu_L(\tilde{\sigma}\setminus \operatorname{Fin}(\tilde{\sigma}))\leq\St(\tilde{\mu}(\tilde{\sigma}\setminus{}^*[-n,n]))$$ for any $n\in \mathbb{N}$. 
	
	Fix $\epsilon\in\mathbb{R}_{>0}$. Using Properties \ref{scale_property_probability} and \ref{scale_property_standard_bias} of Definition \ref{definition_scale}, we have $$\lim_{n\rightarrow\infty}\St\left(\sum_{j=n+1}^{N}\tilde{c}_j\|\tilde{e}_j\|^2\right)=1-\lim_{n\rightarrow\infty}\St\left(\sum_{j=1}^n\tilde{c}_j\|\tilde{e}_j\|^2\right)=0. $$ Thus, we may take $j_0\in\mathbb{N}$ such that $$\St\left(\sum_{j=j_0+1}^N\tilde{c}_j\|\tilde{e}_j\|^2\right)<\epsilon.$$ Then, for any $n\in\mathbb{N}$, we have, using Remark \ref{remark_internal_eigenspaces} to commute $\tilde{A}$ and $\Proj_{\tilde{H}_{\lambda}}$,
	\begin{align*}
		\tilde{\mu}(\tilde{\sigma}\setminus{}^*[-n,n])&=\sum_{j=1}^N\tilde{c}_j\sum_{\lambda\in\tilde{\sigma}\setminus{}^*[-n,n]}\|\Proj_{\tilde{H}_{\lambda}}\tilde{e}_j\|^2\\
		&=\sum_{j=1}^{j_0}\tilde{c}_j\sum_{\lambda\in\tilde{\sigma}\setminus{}^*[-n,n]}\|\Proj_{\tilde{H}_{\lambda}}\tilde{e}_j\|^2+\sum_{j=j_0+1}^N\tilde{c}_j\|\Proj_{\tilde{H}_{\tilde{\sigma}\setminus{}^*[-n,n]}}\tilde{e}_j\|^2\\
		&\leq \sum_{j=1}^{j_0}\tilde{c}_j\sum_{\lambda\in\tilde{\sigma}\setminus{}^*[-n,n]}\frac{\lambda^2}{n^2}\|\Proj_{\tilde{H}_{\lambda}}\tilde{e}_j\|^2+\sum_{j=j_0+1}^N\tilde{c}_j\|\tilde{e}_j\|^2\\ &\leq\epsilon+\sum_{j=1}^{j_0}\tilde{c}_j\sum_{\lambda\in\tilde{\sigma}\setminus{}^*[-n,n]}\frac{\lambda^2}{n^2}\|\Proj_{\tilde{H}_{\lambda}}\tilde{e}_j\|^2\\
		&=\epsilon+\frac{1}{n^2}\sum_{j=1}^{j_0}\tilde{c}_j\sum_{\lambda\in\tilde{\sigma}\setminus{}^*[-n,n]}\|\lambda \Proj_{\tilde{H}_{\lambda}}\tilde{e}_j\|^2\\
		&=\epsilon+\frac{1}{n^2}\sum_{j=1}^{j_0}\tilde{c}_j\sum_{\lambda\in\tilde{\sigma}\setminus{}^*[-n,n]}\|\tilde{A} \Proj_{\tilde{H}_{\lambda}}\tilde{e}_j\|^2\\
		&=\epsilon+\frac{1}{n^2}\sum_{j=1}^{j_0}\tilde{c}_j\sum_{\lambda\in\tilde{\sigma}\setminus{}^*[-n,n]}\|\Proj_{\tilde{H}_{\lambda}}\tilde{A}\tilde{e}_j\|^2\\
		&=\epsilon+\frac{1}{n^2}\sum_{j=1}^{j_0}\tilde{c}_j\|\Proj_{\tilde{H}_{\tilde{\sigma}\setminus{}^*[-n,n]}}\tilde{A}\tilde{e}_j\|^2\\
		&\leq \epsilon+\frac{1}{n^2}\sum_{j=1}^{j_0}\tilde{c}_j\|\tilde{A}\tilde{e}_j\|^2.
	\end{align*}
By Property \ref{scale_compat_property_A} of Definition \ref{definition_scale} and our earlier observation, for any $n\in\mathbb{N}$ we deduce that  $$\mu_L(\tilde{\sigma}\setminus \operatorname{Fin}(\tilde{\sigma}))\leq\epsilon+\frac{1}{n^2}\sum_{j=1}^{j_0}c_j\St(\|\tilde{A}\tilde{e}_j\|)^2.$$
Letting $n\to \infty$, we have that $\mu_L(\tilde{\sigma}\setminus \operatorname{Fin}(\tilde{\sigma}))\leq\epsilon$. Since $\epsilon\in\mathbb{R}_{>0}$ is arbitrary, we conclude that $\mu_L(\tilde{\sigma}\setminus \operatorname{Fin}(\tilde{\sigma}))=0$
\end{proof}

\begin{defn}
    Let $\St_{\tilde{\sigma}}:=\St|_{\operatorname{Fin}(\tilde{\sigma})}:\operatorname{Fin}(\tilde{\sigma})\rightarrow\mathbb{R}.$
\end{defn}
By \cite[Theorem 3.2]{Ross1997}, we have that
    the function $\St_{\tilde{\sigma}}$ is measurable.
We can now define the measure on $\mathbb{R}$ that will be part of the verification of Theorem \ref{theorem_spectral_direct}.

\begin{defn}
	Let $\mu$ be the probability measure on $\operatorname{Borel}(\mathbb{R})$ obtained by pushing forward $\mu_L$ via $\St_{\tilde{\sigma}}.$
\end{defn}

\section{Internal Isometry}
In this section, we consider properties given by the internal direct integral.

\begin{defn}
	Define  $$\tilde{U}:{}^*H\rightarrow\int_{\tilde{\sigma}}^{\oplus}\tilde{H}_{\lambda}d\tilde{\mu}(\lambda)$$ by $(\tilde{U}(x))(\lambda)=\frac{1}{\sqrt{\tilde{\mu}(\lambda)}}\Proj_{\tilde{H}_\lambda}x$.
\end{defn}
\begin{rmk}
	Since $\tilde{\sigma}$ is hyperfinite, $\tilde{\mathcal{A}}={}^*\mathcal{P}(\tilde{\sigma})$, and $\tilde{\mu}(\lambda)>0$ for all $\lambda\in\tilde{\sigma}$, we have that $\int_{\tilde{\sigma}}^{\oplus}\tilde{H}_{\lambda}d\tilde{\mu}(\lambda)=\prod_{\lambda\in\tilde{\sigma}}\tilde{H}_{\lambda}$ (as internal sets). As such, for any $x\in{}^*H,$ we have that $\lambda\rightarrow \frac{1}{\sqrt{\tilde{\mu}(\lambda)}}\Proj_{\tilde{H}_\lambda}x$ is an internal function, and so $\tilde{U}(x)\in$ does indeed belong to $\int_{\tilde{\sigma}}^{\oplus}\tilde{H}_{\lambda}d\tilde{\mu}(\lambda) $.
\end{rmk}

\begin{prop}\label{proposition_internal_isometry}
	$\tilde{U}$ is an internal linear map. Furthermore, $\ker(\tilde{U})=\tilde{H}_{\tilde{\sigma}}^\perp$ and $\tilde{U}|_{\tilde{H}_{\tilde{\sigma}}}$ is an isometry. Finally, for any $\lambda\in\tilde{\sigma}$ and $x\in\tilde{H}$, $(\tilde{U}(\tilde{A}x))(\lambda)=\lambda(\tilde{U}(x))(\lambda)$ holds.
\end{prop}

\begin{proof}
	It is clear that $\tilde{U}$ is internal.  Linearity of $\tilde{U}$ follows from linearity of othogonal projection maps.
	
	Furthermore, for any $x\in{}^*H$, we have that 
	\begin{align*}
		\|\tilde{U}(x)\|_{\tilde{\mu}}^2&=\int_{\tilde{\sigma}}\|(\tilde{U}(x))(\lambda)\|^2d\tilde{\mu}(\lambda)\\
		&=\sum_{\lambda\in \tilde{\sigma}}\left\|\frac{1}{\sqrt{\tilde{\mu}(\lambda)}}\Proj_{\tilde{H}_\lambda}x\right\|^2\tilde{\mu}(\lambda)\\
		&=\sum_{\lambda\in \tilde{\sigma}}\left\|\Proj_{\tilde{H}_\lambda}x \right\|^2\\
		&=\| \Proj_{\tilde{H}_{\tilde{\sigma}}}x\|^2.
	\end{align*}
	Therefore, we have $\ker(\tilde{U})=\tilde{H}_{\tilde{\sigma}}^{\perp}$ and that $\tilde{U}|_{\tilde{H}_{\tilde{\sigma}}}$ is an isometry. The last part of the proposition follows from Remark \ref{remark_internal_eigenspaces}.
\end{proof}

An interesting idea would be to use the internal direct integral above to construct an external direct integral $\int_{\tilde{\sigma}}^{\oplus}\hat{H}_{\lambda}d\mu_L$, where $\hat{H}_{\lambda}$ is the Hilbert space nonstandard hull of $\tilde{H}_{\lambda}$. While such a construction is certainly possible, it presents significant technical difficulties, most notably with the measurable structure of the direct integral. Furthermore, pushing forward that structure to obtain a suitable direct integral $\int_{\mathbb{R}}^{\oplus}H_{\lambda}d\mu$ also presents quite nontrivial difficulties.

Instead, as mentioned in the introduction, the proposed path here is to use the internal direct integral to establish how the interactions between sections should behave, and use these results to construct  $\int_{\mathbb{R}}^{\oplus}H_{\lambda}d\mu$. 

Our next result (Theorem \ref{theorem_nearstandard_sintegrable} below) is crucial moving forward. We first recall the definition of S-integrability.

\begin{defn}
Suppose that $(X,\Omega,\nu)$ is an internal probability space.  Then an internally measurable function $f:X\rightarrow{}^*\mathbb{K}$ is \textbf{S-integrable} if :
    \begin{itemize}
        \item $\int_X|f|d{\nu}$ is finite; and
        \item for all $E\in \Omega$ such that $\nu(E)\simeq0$, we have $\int_{E}|f|d\nu\simeq 0.$
    \end{itemize}
\end{defn}
\begin{rmk}
    Suppose that $f$ as above is S-integrable.  One can then verify that $f$ is $\nu$-almost everywhere finite. Furthermore,   \cite[Theorem 6.2]{Ross1997} establishes that $\operatorname{st}\circ f$ is $\mu_L$-integrable and $$\St\left(\int_Xfd\nu\right)=\int_{X}(\St\circ f) d\nu_L.$$
    Given $B\in \Omega$, it is straightforward to see that $f\cdot \mathbf{1}_B$ is also S-integrable, whence the equaltiy in the above display remains true when integrating over $B$; we shall use this fact a number of times in the sequel.
\end{rmk}
\begin{thm}\label{theorem_nearstandard_sintegrable}
	For all $x,y\in{}^*H$ with $x$ nearstandard and $y$ finite, the map $\langle \tilde{U}(x), \tilde{U}(y)\rangle$ is an S-integrable function on $\tilde{\sigma}$ with respect to $\tilde{\mu}$. 
\end{thm}
\begin{proof}
    First, using the Cauchy-Schwarz inequality twice, it is clear that $$\int_{\tilde{\sigma}}\left|\langle \tilde{U}(x),\tilde{U}(y)\rangle\right|d\tilde{\mu}\leq \int_{\tilde{\sigma}}\|\tilde{U}(x)\|\;\|\tilde{U}(y)\|d\tilde{\mu} \leq\|\tilde{U}(x)\|_{\tilde{\mu}}\;\|\tilde{U}(y)\|_{\tilde{\mu}}\leq \|x\|\|y\|.$$ Thus, $\int_{\tilde{\sigma}}\left|\langle \tilde{U}(x),\tilde{U}(y)\rangle\right|d\tilde{\mu}$ is finite.
    It remains to show that, for any internal $E\subset\tilde{\sigma}$ with $\tilde{\mu}(E)\simeq 0$, we have $$\St\left(\int_{E}|\langle \tilde{U}(x),\tilde{U}(y)\rangle|d\tilde{\mu}\right)=0.$$ We have
	\begin{align*}
		\int_{E}|\langle \tilde{U}(x),\tilde{U}(y)\rangle|d\tilde{\mu}&=\sum_{\lambda\in E}\left|\left\langle\tilde{U}(x),\tilde{U}(y)\right\rangle(\lambda) \right|\tilde{\mu}(\lambda)\\
        &=\sum_{\lambda\in E}\left|\left\langle\frac{1}{\sqrt{\tilde{\mu}(\lambda)}}\Proj_{\tilde{H}_{\lambda}}x,\frac{1}{\sqrt{\tilde{\mu}(\lambda)}}\Proj_{\tilde{H}_{\lambda}}y\right\rangle \right|\tilde{\mu}(\lambda)\\
		&=\sum_{\lambda\in E}\left|\left\langle\Proj_{\tilde{H}_{\lambda}}x,\Proj_{\tilde{H}_{\lambda}}y\right\rangle \right|\\
		&\leq\sum_{\lambda\in E}\|\Proj_{\tilde{H}_{\lambda}}x\|\|\Proj_{\tilde{H}_{\lambda}}y\|\\
		&\leq\sqrt{\sum_{\lambda\in E}\|\Proj_{\tilde{H}_{\lambda}}x\|^2}\sqrt{\sum_{\lambda\in E}\|\Proj_{\tilde{H}_{\lambda}}y\|^2}\\
		&=\|\Proj_{\tilde{H}_{E}}x\|\|\Proj_{\tilde{H}_{E}}y\|\leq\|\Proj_{\tilde{H}_{E}}x\|\|y\|.
	\end{align*}
	Since $y$ is finite, it is sufficient (in fact necessary) to show that $\|\Proj_{\tilde{H}_{E}}x\|$ is infinitesimal. Let $\epsilon\in\mathbb{R}_{>0}.$ Using Property \ref{scale_property_density} of Definition \ref{definition_scale} and the fact that $x$ is nearstandard, we may find $k\in\mathbb{N}$ and $\{a_j\}_{j=1}^k$ in $\mathbb{K}$ such that $\|\sum_{j=1}^{k}a_j \tilde{e}_j-x\|<\epsilon.$  Then we have
	\begin{align*}
		\|\Proj_{\tilde{H}_{E}}x\|&\leq\left\|\Proj_{\tilde{H}_{E}}x-\Proj_{\tilde{H}_{E}}\left(\sum_{j=1}^{k}a_j \tilde{e}_j\right)\right\|+\left\| \Proj_{\tilde{H}_{E}}\left(\sum_{j=1}^{k}a_j \tilde{e}_j\right)\right\|\\
		&\leq\left\|x-\sum_{j=1}^{k}a_j \tilde{e}_j\right\|+\sum_{j=1}^{k}|a_j|\left\|\Proj_{\tilde{H}_{E}}\tilde{e}_j\right\|\\
		&\leq\epsilon+\sum_{j=1}^{k}\frac{|a_j|}{\sqrt{\tilde{c}_j}}\sqrt{\tilde{c}_j\|\Proj_{\tilde{H}_{E}}\tilde{e}_j\|^2}\leq\epsilon+\sum_{j=1}^{k}\frac{|a_j|}{\sqrt{\tilde{c}_j}}\sqrt{\tilde{\mu}(E)}.
	\end{align*} 
	Since $\St(\tilde{\mu}(E))=0$, $\St(\|\Proj_{\tilde{H}_{E}}x\|)\leq\epsilon.$ Since $\epsilon$ is arbitrary, we get that $\|\Proj_{\tilde{H}_{E}}x\|$ is infinitesimal, as desired.
\end{proof}
\begin{rmk}
    In the course of the above proof, we established that, for all $x,y\in {}^*H$, we have $$\int_{\tilde{\sigma}}\left|\langle \tilde{U}(x), \tilde{U}(y)\rangle \right|d\tilde{\mu}\leq \|x\|\|y\|.$$
This important inequality is worth isolating here.
\end{rmk}
\begin{cor}\label{corollary_isometry_infinitesimal_perturbation}
    Suppose that $x_1, x_2, y_1,y_2\in{}^*H$ are all nearstandard.  Further suppose that $\St(x_1)=\St(x_2)$ and $\St(y_1)=\St(y_2)$.  Then $$\St\left(\langle\tilde{U}(x_1), \tilde{U}(y_1)\rangle(\lambda)\right) =\St\left(\langle\tilde{U}(x_2), \tilde{U}(y_2)\rangle(\lambda)\right)$$ for $\mu_L$-almost all $\lambda\in\tilde{\sigma}.$
\end{cor}

\begin{proof}
    By Theorem \ref{theorem_nearstandard_sintegrable}, we know that both $\langle \tilde{U}(x_1),\tilde{U}(y_1)\rangle$ and $\langle \tilde{U}(x_2),\tilde{U}(y_2)\rangle$ are S-integrable. Using the triangle inequality, it is straightforward to see that the internal function $\left|\langle \tilde{U}(x_1),\tilde{U}(y_1)\rangle-\langle \tilde{U}(x_2),\tilde{U}(y_2)\rangle \right|$ is S-integrable as well. Therefore,

    \begin{align*}
        \int_{\tilde{\sigma}}&\left|\St(\langle \tilde{U}(x_1),\tilde{U}(y_1)\rangle)-\St(\langle \tilde{U}(x_2),\tilde{U}(y_2)\rangle) \right|d\mu_L\\ &=\int_{\tilde{\sigma}}\St\left(\left|\langle \tilde{U}(x_1),\tilde{U}(y_1)\rangle-\langle \tilde{U}(x_2),\tilde{U}(y_2)\rangle \right|\right)d\mu_L\\
        &=\St\left(\int_{\tilde{\sigma}}\left|\langle \tilde{U}(x_1),\tilde{U}(y_1)\rangle-\langle \tilde{U}(x_2),\tilde{U}(y_2)\rangle \right|d\tilde{\mu}\right)\\
        &=\St\left(\int_{\tilde{\sigma}}\left|\langle \tilde{U}(x_1),\tilde{U}(y_1-y_2)\rangle+\langle \tilde{U}(x_1-x_2),\tilde{U}(y_2)\rangle \right|d\tilde{\mu}\right)\\
        &\leq\St\left(\int_{\tilde{\sigma}}\left|\langle \tilde{U}(x_1),\tilde{U}(y_1-y_2)\rangle\right|d\tilde{\mu}\right)+\St\left(\int_{\tilde{\sigma}}\left|\langle \tilde{U}(x_1-x_2),\tilde{U}(y_2)\rangle\right|d\tilde{\mu}\right)\\
        &\leq \St\left(\|x_1\|\;\|y_1-y_2\|\right)+\St\left(\|x_1-x_2\|\;\|y_2\|\right)=0.
    \end{align*}
    Therefore, $\int_{\tilde{\sigma}}\left|\St(\langle \tilde{U}(x_1),\tilde{U}(y_1)\rangle)-\St(\langle \tilde{U}(x_2),\tilde{U}(y_2)\rangle)\right|d\mu_L=0$, from which we conclude $$\St\left(\langle\tilde{U}(x_1), \tilde{U}(y_1)\rangle(\lambda)\right) =\St\left(\langle\tilde{U}(x_2), \tilde{U}(y_2)\rangle(\lambda)\right)$$ for $\mu_L$-almost all $\lambda\in\tilde{\sigma},$ completing the proof.
\end{proof}

\section{Hilbert Family Construction}
In this section, we build the spectral family that will be used to complete the proof of Theorem \ref{theorem_spectral_direct}.  We begin by using Theorem \ref{theorem_nearstandard_sintegrable} to construct our direct integral using the expected interactions.

\begin{defn}\label{definition_nu_measure}
	For $x,y\in H$, let $\nu^{x,y}$ be the $\mathbb{K}$-signed measure\footnote{When $\mathbb K=\mathbb C$, a $\mathbb K$-signed measure is simply a complex measure.} on $\operatorname{Borel}(\mathbb{R})$ defined by $$\nu^{x,y}(B)=\int_{\St_{\tilde{\sigma}}^{-1}(B)}\St\left(\left\langle\tilde{U}(x),\tilde{U}(y)\right\rangle\right)d\mu_L.$$
\end{defn} 

Note that $\nu^{x,y}$ is in fact a finite $\mathbb K$-signed measure on $\operatorname{Borel}(\mathbb{R})$.  Indeed, by Theorem \ref{theorem_nearstandard_sintegrable}, $\St\left(\left\langle\tilde{U}(x),\tilde{U}(y)\right\rangle\right)$ is a $\mu_L$-integrable function on $\tilde{\sigma}$, whence the map $$V\mapsto \int_{V}\St\left\langle\tilde{U}(x),\tilde{U}(y)\right\rangle d\mu_L$$ is a finite $\mathbb{K}$-signed measure on $\mathcal{A}_L$; the measure $\nu^{x,y}$ is then the push-forward of this measure along $\St_{\tilde{\sigma}}$. Furthermore, the following is a direct consequence of Corollary \ref{corollary_isometry_infinitesimal_perturbation}.

\begin{lemma}\label{lemma:infinitesimalchange}
    Fix $x,y\in H$ and take $\tilde{x},\tilde{y}\in {}^*H$ such that $\tilde{x}\simeq x$ and $\tilde{y}\simeq y$.  Then $$\nu^{x,y}(B)=\int_{\St_{\tilde{\sigma}}^{-1}(B)}\St\left\langle\tilde{U}(\tilde{x}),\tilde{U}(\tilde{y})\right\rangle d\mu_L.$$
\end{lemma}

\begin{prop}
	For all $x,y\in H$, we have $\nu^{x,y}\ll\mu$.
\end{prop}
\begin{proof}
	Let $B\subset \mathbb{R}$ be any Borel set such that $\mu(B)=0$. Then, by definition of $\mu$, we have $\mu_L(\St_{\tilde{\sigma}}^{-1}(B))=\mu(B)=0$, and thus $$\int_{\St_{\tilde{\sigma}}^{-1}(B)}\St\left\langle\tilde{U}(x),\tilde{U}(y)\right\rangle d\mu_L=0.$$
\end{proof}
\begin{defn}
	For $j,l\in\mathbb{N}$, set $U_{j,l}=\frac{d\nu^{e_j,e_l}}{d\mu}:\mathbb{R}\rightarrow\mathbb{K}$. Furthermore, given $t\in\mathbb{R}$ and $n\in\mathbb{N}$, let $M_n(t)$ be the square matrix of size $n$ whose coefficients in $\mathbb{K}$ are given by $(M^{(n)}(t))_{i,j}=U_{i,j}(t)$.  
\end{defn}

These coefficients $U_{j,l}$ will be the foundation of the family of spaces $\{H_{t}\}_{t\in\mathbb{R}}$, as well as the sections corresponding to each $e_j$.

\begin{prop}\label{proposition_measurable_sections}
	There exists a sequence  $(V_j:\mathbb{R}\rightarrow \ell_2^{\mathbb{K}})_{j\in\mathbb{N}}$ of Borel measurable functions such that, for all $j,l\in\mathbb{N}$ and for $\mu$-almost all $t\in \mathbb{R}$, we have $\langle V_j(t),V_l(t)\rangle=U_{j,l}(t)$. 
\end{prop}

To prove the above proposition, we need a lemma. 
\begin{lemma}
	For $\mu$-almost all $t\in\mathbb{R}$, the matrix $M^{(n)}(t)$ is positive semi-definite for all $n\in\mathbb{N}$.
\end{lemma}
\begin{proof}
	For $n\in \mathbb{N}$, set $S_n=\{t\in\mathbb{R}\;|\;M^{(n)}(t)\text{ is positive semi-definite}\}$; we want to show that $\mu(\bigcap_{n\in\mathbb{N}} S_n)=1.$ 
	
	Fix $n\in\mathbb{N}$ and let $\{v^{(k)}\}_{k\in\mathbb{N}}$ be a dense subset of $\mathbb{K}^n$. By density, we have that $t\in S_n$ if and only if $\langle M^{(n)}(t) v^{(k)},v^{(k)}\rangle\geq 0$ holds for all $k\in\mathbb{N}$. Fix $k\in\mathbb{N}$; for $t\in\mathbb{R}$, we have 
	\begin{align*}
		\langle M^{(n)}(t) v^{(k)},v^{(k)}\rangle=\sum_{j,l=1}^{n}U_{l,j}(t)v^{(k)}_j\overline{v^{(k)}_l}.
	\end{align*}
	As such, for any Borel $B\subseteq \mathbb{R}$, we have
	
	\begin{align*}
		\int_{B}\langle M^{(n)}(t) v^{(k)}&,v^{(k)}\rangle d\mu(t)=\sum_{j,l=1}^{n}v^{(k)}_j\overline{v^{(k)}_l}\int_{B}U_{l,j}d\mu\\
		&=\sum_{j,l=1}^{n}v^{(k)}_j\overline{v^{(k)}_l}\nu^{e_l,e_j}(B)\\
		&=\sum_{j,l=1}^{n}v^{(k)}_j\overline{v^{(k)}_l}\int_{\St^{-1}(B)}\St\left(\left\langle\tilde{U}(e_l),\tilde{U}(e_j)\right\rangle\right)d\mu_L\\
		&=\int_{\St^{-1}(B)}\St\left(\left\|\sum_{j=1}^{n}\overline{v^{(k)}_j}(\tilde{U}\left(e_j\right))(\lambda) \right\|^2\right)d\mu_L(\lambda)\\
		&\geq0.
	\end{align*}
	Since this holds for any Borel $B\subseteq \mathbb{R}$, we have that $\langle M^{(n)}(t) v^{(k)},v^{(k)}\rangle \geq0$ for $\mu$-almost all $t\in \mathbb{R}$. Since this held for an arbitrary $k\in \mathbb{N}$, we have that we have that $\mu(S_n)=1$. Since the latter sentence held for an arbitrary $n\in \mathbb{N}$, we conclude that $\mu(\bigcap_{n\in\mathbb{N}} S_n)=1$. 
\end{proof}
\begin{proof}[Proof of Proposition \ref{proposition_measurable_sections}]
	We begin by setting up some notation to be used throughout the proof.  First, given any Hilbert space $E$, let $R_E: E'\rightarrow E$ be the map given by the Riesz representation Theorem, that is, for $f\in E'$ and $x\in E$, we have that $\langle x,R_E(f)\rangle=f(x)$. Second, we let $(g_1,g_2,\;\dots)$ denote the standard orthonormal basis of $\ell_2^{\mathbb{K}}$.  Finally, we set $$S:=\{t\in\mathbb{R}\;|\; \text{ for all $n\in\mathbb{B}$, } M^{(n)}(t)\text{ is positive semi-definite}\}.$$ 
    
    By the preceding lemma, $\mu(S)=1$.  As a result, we only need to define each $V_i$ on $S$. We define the sequence $(V_i)_{i\in \mathbb{N}}$ by recursion.  Throughout the construction, we recursively assume that, for any $t\in S$, we have $V_n(t)\in\Span\{g_j\}_{j=1}^n$.
	
	Suppose that $V_1,\ldots,V_{n-1}$ have been defined satisfying the conclusion of the proposition and the extra condition stated above.  We now show how to define $V_n$ on $S$.  In the rest of the proof, we assume that $t\in S$.  
	
	Given $a_1,\ldots,a_{n-1}\in \mathbb{K}$, let $a=(\overline{a_1},\ldots,\overline{a_{n-1}},0)\in \mathbb{K}^n$.  Using the Cauchy-Schwarz inequality and the fact that $M^{(n)}(t)$ is positive semi-definite, we have that:
	\begin{align*}
		\left|\sum_{j=1}^{n-1} a_j U_{j,n}(t)\right|^2&=|\langle M^{(n)}(t)(0,\dots,0,1) ,a\rangle|^2\\
		&\leq \langle M^{(n)}(t)(0,\dots,0,1),(0,\dots,0,1)\rangle\cdot \langle M^{(n)}(t)a,a\rangle\\
		&=U_{n,n}(t)\left(\sum_{j,l=1}^{n-1}a_j\overline{a_l}U_{j,l}(t)\right)\\
		&=U_{n,n}(t)\left(\sum_{j,l=1}^{n-1}a_j\overline{a_l}\langle V_j(t),V_l(t)\rangle\right)\\
		&=U_{n.n}(t)\left\|\sum_{j=1}^{n-1}a_jV_j(t)\right\|^2.
	\end{align*}
	Set $E_n(t):=\Span\{V_j(t)\}_{j=1}^{n-1}$.  The above inequality allows us to define the function $T_n(t):E_n(t)\rightarrow\mathbb{K}$ by $$T_n(t)\left(\sum_{j=1}^{n-1}a_jV_j(t)\right)=\sum_{j=1}^{n-1}a_jU_{j,n}(t).$$ It also follows that $T_n(t)$ is a bounded linear functional with $\|T_n(t)\|^2\leq U_{n,n}(t)$. 
	
	Set $$V_n(t)=R_{E_n(t)}(T_n(t))+\sqrt{U_{n,n}(t)-\|T_n(t)\|^2}g_n.$$  Note first that $V_n(t)\in\Span\{g_j\}_{j=1}^n$.  We next show that the inner products of $V_i(t)$ and $V_n(t)$ for $i=1,\ldots,n-1$ are as desired.  To see this, first note that $g_n\perp E_n(t)$. Using that $M^{(n)}(t)$ is self-adjoint, we have, for any $j\in [n-1] $, that:
	
	\begin{align*}
		\langle V_j(t),V_n(t)\rangle&=\langle V_j(t),R_{E_n(t)}(T_n(t))\rangle =T_n(t)(V_j(t))=U_{j,n}(t);\\
		\langle V_n(t),V_j(t)\rangle&=\overline{\langle V_j(t),V_n(t)\rangle}=\overline{U_{j,n}(t)}=U_{n,j}(t);\\
		\langle V_n(t),V_n(t)\rangle &=\|R_{E_n(t)}(T_n(t))\|^2+U_{n,n}(t)-\|T_n(t)\|^2=U_{n,n}(t).
	\end{align*}
	To finish the proof, we must show that $V_n$ is measurable. Since $U_{n,n}$ is measurable (by definition), since $\|T_n(t)\|^2=\|R_{E_n(t)}(T_n(t))\|^2$, and  since $\ell_2^{\mathbb{K}}$ is separable, it follows from Lemma \ref{lem:HSmeasurablefunction} of the Appendix that it is sufficient to show that $$t\rightarrow \langle y,R_{E_n(t)}(T_n(t))\rangle=\langle\Proj_{E_n(t)}y,R_{E_n(t)}(T_n(t))\rangle=(T_n(t))(\Proj_{E_n(t)}y)$$ is measurable for any $y\in\ell_2^{\mathbb{K}}$. By Lemma \ref{lemma_measurable_projection} of the Appendix, there are measurable functions $a_1,\;\dots,\;a_{n-1}:S\rightarrow\mathbb{K}$ such that $$\Proj_{E_n(t)}y=\sum_{j=1}^{n-1} a_j(t) V_j(t). $$ It follows that $$(T_n(t))(\Proj_{E_n(t)}y)=\sum_{j=1}^{n-1} a_j(t) U_{j,n}(t),$$ which is thus measurable, as desired.
\end{proof}
\begin{conv}
    We now fix a sequence $(V_j)_{j\in\mathbb{N}}$ of functions satisfying Proposition \ref{proposition_measurable_sections}. Furthermore, given $t\in\mathbb{R}$, we define $$H_t=\overline{\Span\{V_j(t)\}_{j\in\mathbb{N}}}\leq \ell_2^{\mathbb{K}}.$$
\end{conv}
    


\begin{rmk}
    By Proposition \ref{prop:inducedmeasfamily} of the Appendix, we have an induced measurable structure on $\prod_{t\in \mathbb{R}}H_t$ whose measurable sections are those elements $X\in \prod_{t\in \mathbb{R}} H_t$ for which $X$ is measurable when viewed as a function $\mathbb{R}\to \ell_2^{\mathbb{K}}$.
\end{rmk}

	

We can now establish the existence of the map from Theorem \ref{theorem_spectral_direct}.

\begin{prop}
	There exists a unique linear isometry $U:H\rightarrow\int_{\mathbb{R}}^{\oplus}H_td\mu$ such that for each $j\in\mathbb{N},$ $U(e_j)=V_j$. 
\end{prop}
\begin{proof}
	For each $j,l\in\mathbb{N}$, we have
	\begin{align*} \langle V_j,V_l\rangle_{\mu}&=\int_{\mathbb{R}}\langle V_j,V_l\rangle d\mu=\int_{\mathbb{R}}U_{j,l}d\mu=\nu^{e_j,e_l}(\mathbb{R})\\
	&=\int_{\operatorname{Fin}({}^*\mathbb{R})\cap \tilde{\sigma}}\St\left(\left\langle\tilde{U}(\tilde{e}_j),\tilde{U}(\tilde{e}_l)\right\rangle\right)d\mu_L\\
    &=\int_{\tilde{\sigma}}\St\left(\left\langle\tilde{U}(\tilde{e}_j),\tilde{U}(\tilde{e}_l)\right\rangle\right)d\mu_L\\
	&=\St\left(\int_{\tilde{\sigma}}\left\langle\tilde{U}(\tilde{e}_j),\tilde{U}(\tilde{e}_l)\right\rangle d\tilde{\mu}\right)=\St(\langle\tilde{U}(\tilde{e}_j),\tilde{U}(\tilde{e}_l)\rangle_{\tilde{\mu}})\\
	&=\St(\langle\tilde{e_j},\tilde{e}_l\rangle)=\langle e_j,e_l\rangle.
	\end{align*}
    Note that the fourth equality used Lemma \ref{lemma:infinitesimalchange}, while the fifth followed from Proposition \ref{proposition_almostall_eigenvalues_finite} and the sixth used Theorem \ref{theorem_nearstandard_sintegrable}.\footnote{In the sequel, we will not comment on such applications of these results.}
	A particular consequence of this calculation is that each $V_j$ does indeed belong to  $\int_{\mathbb{R}}^{\oplus}H_td\mu(t)$. Another consequence is that, for any  sequence $a_1,\;\dots\;,a_k$ in $\mathbb{K}$, we have $\left\|\sum_{j=1}^k a_je_j\right\|=\left\|\sum_{j=1}^k a_jV_j\right\|_{\mu}.$ From that, we conclude that there is a unique linear isometry $U:\Span(\{e_j\}_{j\in\mathbb{N}})\rightarrow\int_{\mathbb{R}}^{\oplus}H_td\mu(t)$ such that for any $j$, $U(e_j)=V_j$. Furthermore, this isometry uniquely extends to all of $H$ by Property (\ref{scale_property_density}) of Definition \ref{definition_scale}.
    \end{proof}


We have almost completed the proof of Theorem \ref{theorem_spectral_direct}.  The only thing remaining to show is that $U$ indeed ``intertwines'' $A$ and the multiplication operator on the direct integral.  Note that, for all $j,l\in \mathbb{N}$, we have $\frac{d\nu^{e_j,e_l}}{d\mu}=\langle U(e_j),U(e_l)\rangle$.  The next proposition, which is key for linking $U$ and $A$, shows that this remains true for all $x,y\in H$.

\begin{prop}\label{proposition_isometry_derivatives}
    For all $x,y\in H$, we have $$\frac{d\nu^{x,y}}{d\mu}=\langle U(x), U(y)\rangle.$$
\end{prop}
	
	\begin{proof}
    Fix $x,y\in H$ and Borel $B\subseteq \mathbb{R}$.  We must show that $$\nu^{x,y}(B)=\int_B \langle U(x),U(y)\rangle d\mu.$$ Fix $\epsilon\in\mathbb{R}_{>0}$. Then, given some large enough $k$, let $x_{\epsilon}=\sum_{j=1}^k a_je_j$ and $y_{\epsilon}=\sum_{j=1}^k b_je_j$ such that $\|x-x_{\epsilon}\|<\frac{\epsilon}{4(\|y\|+1)}$ and  $\|y-y_{\epsilon}\|<\frac{\epsilon}{4(\|x\|+\epsilon)}.$ Then, we have
	\begin{align*}
		&\left|\nu^{x,y}(B)-\sum_{j,l=1}^{k}a_j\overline{b_l}\nu^{e_j,e_l}(B)\right|\\
		&=\left|\int_{\St^{-1}(B)}\St\left(\langle\tilde{U}(x),\tilde{U}(y)\rangle-\sum_{j,l=1}^{k}a_j\overline{b_l}\langle\tilde{U}(e_j),\tilde{U}(e_l)\rangle\right)d\mu_L\right|\\
		&\leq\St\left(\int_{\tilde{\sigma}}\left|\langle\tilde{U}(x),\tilde{U}(y)\rangle-\left\langle\tilde{U}(x_{\epsilon}),\tilde{U}(y_{\epsilon})\right\rangle\right|d\tilde{\mu}\right)\\
		&\leq\St\left(\int_{\tilde{\sigma}}\left|\left\langle\tilde{U}(x-x_{\epsilon}),\tilde{U}(y)\right\rangle\right|d\tilde{\mu}\right)+\St\left(\int_{\tilde{\sigma}}\left|\left\langle\tilde{U}(x_{\epsilon}),\tilde{U}(y-y_{\epsilon})\right\rangle\right|d\tilde{\mu}\right)\\
		&\leq\St\left(\|\tilde{U}(x-x_{\epsilon})\|_{\tilde{\mu}}\|\tilde{U}(y)\|_{\tilde{\mu}}\right)+\St\left(\|\tilde{U}(x_{\epsilon})\|_{\tilde{\mu}}\|\tilde{U}(y-y_{\epsilon})\|_{\tilde{\mu}}\right)\\
		&\leq\|x-x_{\epsilon}\|\;\|y\|+\|x_{\epsilon}\|\;\|y-y_{\epsilon}\|<\frac{1}{2}\epsilon.
	\end{align*}
	On the other hand, we have
	\begin{align*}
		&\left|\int_B\langle U(x),U(y)\rangle d\mu-\sum_{j,l=1}^{k}a_j\overline{b_l}\nu^{e_j,e_l}(B)\right|\\
		&=\left|\int_B\langle U(x),U(y)\rangle d\mu-\sum_{j,l=1}^{k}a_j\overline{b_l}\int_B\langle V_j,V_l\rangle d\mu\right|\\
		&=\left|\int_B\left(\langle U(x),U(y)\rangle-\langle\sum_{j=1}^{k}a_jU(e_j),\sum_{j=1}^{k}b_jU(e_j)\rangle\right) d\mu\right|\\
		&\leq\int_{\mathbb{R}}\left|\langle U(x),U(y)\rangle-\langle U(x_{\epsilon}),U(y_{\epsilon})\rangle \right|d\mu\\
		&\leq\int_{\mathbb{R}}\left|\langle U(x-x_{\epsilon}),U(y)\rangle\right|d\mu+\int_{\mathbb{R}}\left|\langle U(x_{\epsilon}),U(y-y_{\epsilon})\rangle\right|d\mu\\
        &\leq\|U(x-x_{\epsilon})\|_{\mu}\|U(y)\|_{\mu}+\|U(x_{\epsilon})\|_{\mu}\|U(y-y_{\epsilon})\|_{\mu}\\
		&=\|x-x_{\epsilon}\|\|y\|+\|x_{\epsilon}\|\|y-y_{\epsilon}\|<\frac{1}{2}\epsilon.
	\end{align*}
	It follows that $|\int_B\langle U(x),U(y)\rangle d\mu-\nu^{x,y}(B)|<\epsilon$. Since $\epsilon$ is arbitrary, we have $\int_B\langle U(x),U(y)\rangle d\mu=\nu^{x,y}(B)$, completing the proof.
\end{proof}

We can now prove the final piece of Theorem \ref{theorem_spectral_direct}.

\begin{thm}\label{theorem_conclude_spectral}
	For any $x\in\Dom(A)$, we have $(U(Ax))(t)=t\cdot (U(x))(t)$ for $\mu$-almost all $t\in\mathbb{R}$. In other words, for any such $x\in\Dom(A)$, $\Id_{\mathbb{R}}\cdot U(x)= U(Ax).$
\end{thm}
\begin{proof}
	Since $\{V_j(t)\}_{j\in\mathbb{N}}$ is dense-spanning in $H_t$, we only need to show that, for all $j\in\mathbb{N}$, we have $\langle(U(Ax))(t),V_j(t)\rangle=\langle t(U(x))(t),V_j(t)\rangle$ for $\mu$-almost all $t\in \mathbb{R}$. Thus, it is sufficient to show, for any bounded Borel set $B\subseteq \mathbb{R}$, that $$\int_B\left(\langle U(Ax),V_j\rangle-\Id_{\mathbb{R}}\cdot\langle U(x),V_j\rangle\right)d\mu=0,$$ noting that the function is indeed integrable since $B$ is bounded. However, by Proposition \ref{proposition_internal_isometry}, Corollary \ref{corollary_isometry_infinitesimal_perturbation}, and Proposition \ref{proposition_isometry_derivatives}, we have that
	\begin{align*}
		\int_B \Id_{\mathbb{R}}\cdot\langle U(x),V_j\rangle d\mu&=\int_B \Id_{\mathbb{R}}d\nu^{x,e_j}\\ 
        &=\int_{\St^{-1}(B)}(\Id_{\mathbb{R}}\circ\St_{\tilde{\sigma}})\cdot\St(\langle\tilde{U}(x),\tilde{U}(e_j)\rangle)d\mu_L \\
        &=\int_{\St^{-1}(B)}\St(\lambda)\cdot\St(\langle\tilde{U}(\bar{x}),\tilde{U}(\tilde{e}_j)\rangle(\lambda))d\mu_L(\lambda)\\
		&=\int_{\St^{-1}(B)}\St(\lambda\langle(\tilde{U}(\bar{x}))(\lambda),(\tilde{U}(\tilde{e}_j))(\lambda)\rangle)d\mu_L(\lambda)\\
		&=\int_{\St^{-1}(B)}\St(\langle(\tilde{U}(\tilde{A}\bar{x}))(\lambda),(\tilde{U}(\tilde{e}_j))(\lambda)\rangle)d\mu_L(\lambda)\\
		&=\int_{\St^{-1}(B)}\St(\langle \tilde{U}(\tilde{A}\bar{x}),\tilde{U}(\tilde{e}_j)\rangle)d\mu_L=\nu^{Ax,e_j}(B)\\
		&=\int_B\langle U(Ax),V_j\rangle d\mu.
	\end{align*}
	Here, $\bar{x}$ is given by Property (\ref{sampling_property_approximation}) of Definition \ref{definition_sampling} applied to $x$.  Therefore, as stated, we conclude that $t(U(x))(t)=(U(Ax))(t)$ for $\mu$-almost all $t\in\mathbb{R}$, concluding the proof. 
\end{proof}

\section{Self-Adjointness and Surjectivity of the Isometry}
Up until now, we have only assumed that $A$ is a symmetric operator on $H$. As a result, the resulting isometry $U$ cannot be expected to be unitary if $A$ does not have a self-adjoint extension. A significant point of departure from the results given in \cite{nonez2024spectralequivalencesnonstandardsamplings} is that the mere fact that $A$ is self-adjoint will guarantee that the resulting $U$ is surjective. In this section, we in fact prove a bit more.

First, we establish a general fact about quasi-samplings.

\begin{prop}
	We have $\St(G(\tilde{A}))\subset H\times H$ is the graph of a closed symmetric extension of $A$.
\end{prop}
\begin{proof}
 It is clear that $\St(G(\tilde{A}))$ is a linear $\mathbb{K}$-subspace of $H\times H$. Furthermore, since ${}^*$ is $\aleph_1$-saturated, $\St(G(\tilde{A}))$ is closed. We next show that it is the graph of an operator. 
 
 First, to show that it is the graph of a function, take $(x,y_1),(x, y_2)\in \St(G(\tilde{A}))$; we must show that $y_1=y_2$. Take $\bar{x}_1,\bar{x}_2 \in\tilde{H}$ such that $\St((\bar{x}_1,\tilde{A}\bar{x}_1))=(x,y_1)$ and $\St((\bar{x}_2,\tilde{A}\bar{x}_2))=(x,y_2)$. Fix $z\in\Dom(A)$ and take $\bar{z}\in\tilde{H}$ satisfying Property (\ref{sampling_property_approximation}) of Definition \ref{definition_sampling}; we then have 
 \begin{align*}
 	\langle y_1,z\rangle=\St(\langle\tilde{A}\bar{x}_1,\bar{z}\rangle)=\St(\langle\bar{x}_1,\tilde{A}\bar{z}\rangle)=\langle x,Az\rangle=\St(\langle\bar{x}_2,\tilde{A}\bar{z}\rangle)=\St(\langle\tilde{A}\bar{x}_2,\bar{z}\rangle)=\langle y_2,z\rangle.
 \end{align*}
 Since $\Dom(A)$ is dense in $H$, we conclude that $y_1=y_2$. Thus, $\St(G(\tilde{A}))$ is the graph of a function, which we call $\hat{A}$.  Since $A\subset \hat{A}$ by Property (\ref{sampling_property_approximation}) of Definition \ref{definition_sampling}, we know that $\hat{A}$ is densely defined, and since $G(\hat{A})=\St(G(\tilde{A}))$ is a linear closed subspace of $H\times H$, we know that $\hat{A}$ is a linear and closed operator. 
 
 To conclude the proof, we show that $\hat{A}$ is symmetric. Fix $x,z \in\Dom(\hat{A})$ and take $\bar{x}, \bar{z}\in \tilde{H}$ such that $\St((\bar{x},\tilde{A}\bar{x}))=(x,\hat{A}x)$ and $\St((\bar{z},\tilde{A}\bar{z}))=(z,\hat{A}z).$ We then have that
 \begin{align*}
 	\langle\hat{A}x,z\rangle=\St(\langle\tilde{A}\bar{x},\bar{z}\rangle)=\St(\langle\bar{x},\tilde{A}\bar{z}\rangle)=\langle x,\hat{A}z\rangle.
 \end{align*}
We conclude that $\hat{A}$ is symmetric.
\end{proof}

In what follows, we continue to let $\hat{A}$ denote the closed symmetric extension of $A$ whose graph is given by $\St(G(\tilde{A}))$.

\begin{rmk}\label{remark_sampling_extension}
	Note that $(\tilde{H}, \tilde{A})$ is also a quasi-sampling for $\hat{A}$.
\end{rmk}

We can now state the main result of this section.
\begin{thm}\label{theorem_surjectivity}
	Letting $U:H\rightarrow\int_{\mathbb{R}}^{\oplus}H_td\mu$ be the isometry defined above, we have that $U$ is surjective if and only if $\hat{A}$ is self-adjoint.
\end{thm}
\begin{rmk}
	As a consequence of the previous theorem, we have that if $A$ is essentially self-adjoint, then any choice of the quasi-sampling and scale induces that $\hat{A}=\overline{A}$ and $U$ is surjective. 
\end{rmk}
Before we begin the proof of the previous theorem, we need a definition.
\begin{defn}
	For a measurable function $f:\mathbb{R}\rightarrow\mathbb{R}$, we define the operator $T_f$ on $\int_{\mathbb{R}}^{\oplus}H_td\mu(t)$ with $\Dom(T_f)=\{X\in\int_{\mathbb{R}}^{\oplus}H_td\mu(t)\;:\; \int_{\mathbb{R}}|f(t)|^2\|X(t)\|^2d\mu(t)<\infty  \}$ by $(T_f(X))(t)=f(t)\cdot X(t)$. Furthermore, we set $T=T_{\Id_{\mathbb{R}}}$.
\end{defn}
\begin{rmk}
	It is well-known that, for any measurable function $f$, $T_f$ is a self-adjoint operator. 
\end{rmk}
\begin{rmk}
	The statement of Theorem \ref{theorem_conclude_spectral} can now be stated as $U\circ A \subset T\circ U$. We can conclude from Remark \ref{remark_sampling_extension} that $U\circ \hat{A} \subset T\circ U$ holds as well.
\end{rmk}
\begin{proof}[Proof of Theorem \ref{theorem_surjectivity}]
	Fist suppose that $U$ is surjective; we show that $\hat A$ is self-adjoint.  Set $\hat{T}=U^{-1}\circ T\circ U$ as an operator on $H$. Since $U$ is unitary and $T$ is self-adjoint, we have that $\hat{T}$ is self-adjoint. Furthermore, from the previous remark, we have that $\hat{A}\subset \hat{T}.$ Thus, to conclude that $\hat{A}$ is self-adjoint, it is sufficient to show that $\Dom(\hat{A})\supset \Dom(\hat{T})=U^{-1}(\Dom(T)).$
	
	Suppose that $x\in U^{-1}(\Dom(T))$, so that $$\int_{\mathbb{R}}|t|^2\|(U(x))(t)\|^2d\mu<\infty.$$  We aim to show that $x\in \Dom(\hat A)$.
	
	Let $(a_n)_{n\in\mathbb{Z}}$ be an integer-indexed sequence in $\mathbb{R}$ such that, for each $n\in\mathbb{Z}$,  we have $a_n\in[n,n+1]$ and $\mu(\{a_n\})=0$. For $n\in\mathbb{N}$, set $$x_n=(U^{-1}\circ T_{\mathbf{1}_{[a_{-n},a_n]}}\circ U)(x),$$ so that $U(x_n)=\mathbf{1}_{[a_{-n},a_n]}\cdot U(x)$. Since $\int_{\mathbb{R}}|\Id_{\mathbb{R}}|^2\|U(x_n)\|^2d\mu\leq(n+1)^2\|U(x_n)\|_{\mu}^2$, we have that $x_n\in\Dom(T\circ U).$ Furthermore, by the Monotone Convergence Theorem, we have that $x_n\rightarrow x$ and $\hat{T}x_n\rightarrow \hat{T}x$ as $n\rightarrow\infty.$ Thus, since $\hat{A}$ is closed, in order to conclude that $x\in\Dom(\hat{A})$, it is sufficient to show that $(x_n, \hat{T}x_n)\in G(\hat{A})$ for all $n\in \mathbb{N}$.
	
	Towards that end, fix $n\in \mathbb{N}$ and set $I_n={}^*[a_{-n},a_n]\cap\tilde{\sigma}$ and $\bar{x}_n=\Proj_{\tilde{H}_{I_n}}x_n\in\tilde{H}.$ It will suffice to show that $\bar{x}_n\simeq x_n$ and $\tilde{A}\bar{x}_n\simeq \hat{T}x_n$.  Indeed, if we prove these two facts, we will have $(x_n,\hat{T}x_n)\in\St(G(\tilde{A}))=G(\hat{A})$, as desired.
    
    To see that $\bar{x}_n\simeq x_n$, first note that $I_n\subseteq \St_{\tilde{\sigma}}^{-1}([a_{-n},a_n])\subseteq I_n\cup\; \St_{\tilde{\sigma}}^{-1}(\{a_{-n},a_n\}) $, whence $I_n$ and $\St_{\tilde{\sigma}}^{-1}([a_{-n},a_n])$ have a $\mu_L$-null-measure symmetric difference.  As a result, we have
	\begin{align*}
		\St\left(\|\bar{x}_n\|^2\right)&=\St\left(\sum_{\lambda\in I_n}\|\Proj_{\tilde{H}_{\lambda}}x_n\|^2\right)\\
		&=\St\left(\int_{I_n}\|\tilde{U}(x_n)\|^2d\tilde{\mu}\right)\\
		&=\int_{I_n}\St\left(\|\tilde{U}(x_n)\|^2\right)d\mu_L\\
		&=\int_{\St_{\tilde{\sigma}}^{-1}([a_{-n},a_n])}\St\left(\|\tilde{U}(x_n)\|^2\right)d\mu_L\\
		&=\nu^{x_n,x_n}([a_{-n},a_n])\\
		&=\int_{[a_{-n},a_n]}\|U(x_n)\|^2d\mu\\
		&=\int_{\mathbb{R}}\|U(x_n)\|^2d\mu=\|U(x_n)\|_{\mu}^2=\|x_n\|^2.
	\end{align*}
	Thus, since $\|x_n\|^2=\|x_n-\bar{x}_n\|^2+\|\bar{x}_n\|^2$, we have that $\bar{x}_n\simeq x_n.$
	
	We next show that $\tilde{A}\bar{x}_n\simeq \hat{T}x_n$. To do this, it is sufficent to show that:  (1) $\St(\|\tilde{A}\bar{x}_n\|^2)=\|\hat{T}x_n\|^2$, and (2) $\St(\langle \tilde{A}\bar{x}_n,y\rangle)=\langle\hat{T}x_n,y\rangle$ for all $y\in H$. Indeed, (1) and (2) imply that $$\|\tilde{A}\bar{x}_n\|^2=\|\tilde{A}\bar{x}_n-\hat{T}x_n\|^2+\|\hat{T}x_n\|^2+2\operatorname{Re}(\langle\tilde{A}\bar{x}_n-\hat{T}x_n,\hat{T}x_n\rangle),$$ which leads to $0=\St(\|\tilde{A}\bar{x}_n-\hat{T}x_n\|^2)$ by taking standard parts on both sides.
	
	To prove (1), we note that since $|\Id_{\tilde{\sigma}}|^2\cdot \mathbf{1}_{I_n}$ is always bounded by $(n+1)^2$, the function $|\Id_{\tilde{\sigma}}|^2\cdot\mathbf{1}_{I_n}\cdot\|\tilde{U}(x_n)\|^2$ is S-integrable. As such, we have:
\begin{align*}
	\St\left(\|\tilde{A}\bar{x}_n\|^2\right)&=\St\left(\sum_{\lambda\in I_n}|\lambda|^2 \|\Proj_{\tilde{H}_{\lambda}}x_n\|^2\right)\\
	&=\St\left(\int_{I_n}|\Id_{\tilde{\sigma}}|^2\cdot\|\tilde{U}(x_n)\|^2d\tilde{\mu}\right)\\
	&=\int_{I_n}\St\left(|\Id_{\tilde{\sigma}}|^2\cdot\|\tilde{U}(x_n)\|^2\right)d\mu_L\\
	&=\int_{\St_{\tilde{\sigma}}^{-1}([a_{-n},a_n])}|\St_{\tilde{\sigma}}|^2\cdot\St\left(\|\tilde{U}(x_n)\|^2\right)d\mu_L\\
	&=\int_{[a_{-n},a_n]}|\Id_{\mathbb{R}}|^2d\nu^{x_n,x_n}\\
	&=\int_{[a_{-n},a_n]}|\Id_{\mathbb{R}}|^2\cdot\|U(x_n)\|^2d\mu\\
	&=\int_{\mathbb{R}}\|T(U(x_n))\|^2d\mu=\|T(U(x_n))\|_{\mu}^2=\|\hat{T}x_n\|^2.
\end{align*}
To prove (2), fix $y\in H$. For the same reasons as earlier, we have that the internal map $\Id_{\tilde{\sigma}}\cdot\mathbf{1}_{I_n}\cdot\langle \tilde{U}(x_n),\tilde{U}(y)\rangle$ is S-integrable. Thus,
\begin{align*}
	\St\left(\langle\tilde{A}\bar{x}_n,y\rangle\right)&=\St\left(\sum_{\lambda\in I_n}\lambda \langle\Proj_{\tilde{H}_{\lambda}}x_n,y\rangle\right)\\&=\St\left(\sum_{\lambda\in I_n}\lambda \langle\Proj_{\tilde{H}_{\lambda}}x_n,\Proj_{\tilde{H}_{\lambda}}y\rangle\right)\\
	&=\St\left(\int_{I_n}\Id_{\tilde{\sigma}}\cdot\langle \tilde{U}(x_n),\tilde{U}(y)\rangle d\tilde{\mu}\right)\\
	&=\int_{I_n}\St\left(\Id_{\tilde{\sigma}}\cdot\langle \tilde{U}(x_n),\tilde{U}(y)\rangle\right)d\mu_L\\
	&=\int_{\St_{\tilde{\sigma}}^{-1}([a_{-n},a_n])}\St_{\tilde{\sigma}}\cdot\St\left(\langle \tilde{U}(x_n),\tilde{U}(y)\rangle\right)d\mu_L\\
	&=\int_{[a_{-n},a_n]}\Id_{\mathbb{R}}d\nu^{x_n,y}\\
	&=\int_{[a_{-n},a_n]}\Id_{\mathbb{R}}\cdot\langle U(x_n),U(y)\rangle d\mu\\
	&=\int_{\mathbb{R}}\langle T(U(x_n)), U(y)\rangle d\mu\\
	&=\langle T(U(x_n)),U(y)\rangle_{\mu}=\langle\hat{T}x_n,y\rangle.
\end{align*}

This concludes the proof of the forwards direction of the theorem.  For the converse direction, suppose that $\hat{A}$ is self-adjoint; we aim to show that $U$ is surjective.  We begin with the following:

\

\noindent \textbf{Claim 1:}  To prove that $U$ is surjective, it suffices to show that, for any $a,b\in \mathbb{R}$ with $a<b$, we have $\mathbf{1}_{[a,b]}\cdot X\in U(H)$ whenever $X\in U(H)$.

\

\noindent \textbf{Proof of Claim 1:}  Since $V_j=U(e_j)$ for all $j\in \mathbb{N}$, it suffices to establish that $\{\mathbf{1}_{[a,b]}\cdot V_j\;|\; j\in\mathbb{N},\; a,b\in\mathbb{R}\; a<b\}$ spans a dense subset of $\int_{\mathbb{R}}^{\oplus}H_td\mu(t)$.  To see this, suppose that $X\in\int_{\mathbb{R}}^{\oplus}H_td\mu(t) $ is such that $\langle X,\mathbf{1}_{[a,b]}\cdot V_j\rangle_{\mu}=0$ for any $j\in\mathbb{N}$ and $a,b\in\mathbb{R}$ with $a<b$; we need to show that $X=0$. For any $j\in \mathbb{N}$, we then have, for all closed intervals $[a,b]\subseteq \mathbb{R}$, that

$$\int_{[a,b]}\langle X,V_j\rangle d\mu=0.$$

Since the collection of closed intervals generates the Borel $\sigma$-algebra on $\mathbb{R}$, we have that $\langle X(t),V_j(t)\rangle=0$ for $\mu$-almost all $t\in \mathbb{R}$. Thus, there is a Borel set $M\subset\mathbb{R}$ such that $\mu(\mathbb{R}\setminus M)=0$ and such that, for all $t\in M$ and for all $j\in\mathbb{N}$, we have $\langle X(t),V_j(t)\rangle=0.$ By definition of $H_t$, we have that for all $t\in M$, we have $X(t)=0$. It follows that $X=0$, as desired, finishing the proof of the claim.

We now fix $X\in U(H)$ and $a,b\in\mathbb{R}$ with $a<b$; we aim to show that $\mathbf{1}_{[a,b]}\cdot X\in U(H)$.  For $n\in\mathbb{N},$ set $Z_n=\mathbf{1}_{[-n,n]}\cdot X$ and $X_n=\Proj_{U(H)}Z_n$. We note that $$\lim_{n\rightarrow\infty} \mathbf{1}_{[a,b]}\cdot X_n={1}_{[a,b]}\cdot\Proj_{U(H)}\lim_{n\rightarrow\infty}Z_n={1}_{[a,b]}\cdot\Proj_{U(H)}X={1}_{[a,b]}\cdot X;$$ since $U(H)$ is closed, it is sufficient to show that, for each $n\in\mathbb{N}$, $\mathbf{1}_{[a,b]}\cdot X_n\in U(H).$ 

Fix $n\in\mathbb{N}$.  Since $Z_n$ is supported on $[-n,n]$, it is clear that $Z_n\in\Dom(T^k)$ for all $k\in\mathbb{N}$. 

\

\noindent \textbf{Claim 2:}  For all $k\in \mathbb{Z}_{\geq0}$, we have $X_n\in\Dom(T^k)$, and that $$T^k X_n=\Proj_{U(H)}T^kZ_n.$$

\

\noindent \textbf{Proof of Claim 2:}  We prove the claim by induction on $k$.  The case $k=0$ is trivial, since $T^0=I$ and $X_n=\Proj_{U(H)}Z_n$ by definition. Suppose now that $X_n\in\Dom(T^k)$ and $T^k X_n=\Proj_{U(H)}T^kZ_n$. We show $X_n\in\Dom(T^{k+1})$ and $T^{k+1}X_n=\Proj_{U(H)}T^{k+1}Z_n$. 

Set $Z=T^k Z_n$. Since $T^kX_n=\Proj_{U(H)}Z$ by hypothesis, what we want to show is that $\Proj_{U(H)}Z\in\Dom(T)$ and that $T(\Proj_{U(H)}Z)=\Proj_{U(H)}TZ.$ 

Take $z,w\in H$ such that $U(z)=\Proj_{U(H)}Z$ and $U(w)=\Proj_{U(H)}(TZ).$ Then, for all $y\in\Dom(\hat{A}),$
\begin{align*}
	\langle z,\hat{A}y\rangle&=\langle U(z),U(\hat{A}y)\rangle=\langle\Proj_{U(H)}Z,U(\hat{A}y)\rangle=\langle Z,U(\hat{A}y)\rangle=\langle Z,T(U(y))\rangle\\&=\langle TZ,U(y)\rangle=\langle \Proj_{U(H)}(TZ),U(y)\rangle=\langle U(w),U(y)\rangle=\langle w,y\rangle.
\end{align*} 
Thus, since $\hat{A}$ is self-adjoint, we have that $z\in\Dom(\hat{A})$, and $\hat{A}z=w.$ Therefore, $\Proj_{U(H)}Z=U(z)\in\Dom(T)$ and $$T(\Proj_{U(H)}Z)=T(U(z))=U(\hat{A}z)=U(w)=\Proj_{U(H)}(TZ).$$
This finishes the proof of the claim.

By the claim and the definition of $T$, it follows that for any real polynomial $p$, we have that $$p\cdot X_n=\Proj_{U(H)}(p\cdot Z_n)\in U(H).$$

Let $(p_k)_{k\in\mathbb{N}}$ be a sequence of real polynomials such that:
\begin{itemize}
	\item for any $t\in\mathbb{R},$ $\lim_{k\rightarrow\infty}p_k(t)=\mathbf{1}_{[a,b]}(t)$;
	\item for any $t\in[-n,n]$ and $k\in\mathbb{N},$ $|p_k(t)|\leq 2$.
\end{itemize}
To obtain such a sequence, consider, for $k\in\mathbb{N},$ the function $c_k:\mathbb{R}\rightarrow\mathbb{R}$ given by $$c_k(t)=\mathbf{1}_{[a,b]}(t) +(1+k(t-a))\mathbf{1}_{(a-\frac{1}{k},a)}+(1+k(b-t))\mathbf{1}_{(b,b+\frac{1}{k})}.$$ We have that each $c_k$ is continuous and with values in $[0,1]$,  while the sequence pointwise converges to $\mathbf{1}_{[a,b]}$. Thus, by the Stone-Weierstrass Theorem, we can take $p_k$ such that $|p_k(t)-c_k(t)|\leq\frac{1}{k}$ for all $t\in [-n-k,n+k].$

By  the Dominated Convergence Theorem and the fact that $Z_n$ is supported on $[-n,n]$, we have that $\lim_{k\rightarrow\infty}(p_k\cdot Z_n)=\mathbf{1}_{[a,b]}\cdot Z_n$ in $\int_{\mathbb{R}}^{\oplus}H_td\mu(t).$ Since projections are continuous, we have that $$\lim_{k\rightarrow\infty}(p_k\cdot X_n)=\Proj_{U(H)}(\mathbf{1}_{[a,b]}\cdot Z_n).$$ In other words, we have $$\lim_{k\rightarrow\infty}\int_{\mathbb{R}}\left\|(p_k\cdot X_n)(t)-(\Proj_{U(H)}(\mathbf{1}_{[a,b]}\cdot Z_n))(t) \right\|^2d\mu(t)=0,$$ and so there exists a subsequence $(p_{k_l})_{l\in\mathbb{N}}$ such that $$\lim_{l\rightarrow\infty}\|(p_{k_l}\cdot X_n)(t)-(\Proj_{U(H)}(\mathbf{1}_{[a,b]}\cdot Z_n))(t)\|^2=0$$ and thus $$\lim_{l\rightarrow\infty}(p_{k_l}\cdot X_n)(t)=(\Proj_{U(H)}(\mathbf{1}_{[a,b]}\cdot Z_n))(t)$$  in $H_t$ for $\mu$-almost all $t\in \mathbb{R}$. However, for all $t\in \mathbb{R}$, we have $$\lim_{l\rightarrow\infty}(p_{k_l}\cdot X_n)(t)=\mathbf{1}_{[a,b]}(t)\cdot X_n(t).$$ Thus, we have that $\Proj_{U(H)}(\mathbf{1}_{[a,b]}\cdot Z_n)=\mathbf{1}_{[a,b]}\cdot X_n$ in $\int_{\mathbb{R}}^{\oplus}H_td\mu(t).$ Therefore, $\mathbf{1}_{[a,b]}\cdot X_n\in U(H),$ concluding the proof that $U$ is surjective.
\end{proof}


\section{The Spectral measure version of the Spectral Theorem}
In this section, we establish a relationship between the spectral measure of a self-adjoint operator $A$ and the measures $\nu^{x,y}$ constructed above, yielding a proof of the spectral measure version of the Spectral Theorem.  We begin by recalling some basic definitions and facts:
\begin{defn}
    A \textbf{projection-valued measure} is a function $P$ from the Borel subsets of $\mathbb{R}$ to the set of projections on $H$ such that:
    \begin{enumerate}
        \item $P(\emptyset)=0$ and $P(\mathbb{R})=I$;
        \item $P(V\cap W)=P(V)P(W)$; and
        \item $P(\bigcup_{j\in\mathbb{N}}V_j)=\sum_{j\in\mathbb{N}}P(V_j)$ whenever the $V_j$ are pairwise disjoint.
    \end{enumerate}

\end{defn}

     Projection-valued measures are described in detail in \cite[Chapter IX]{conway2019course}. We note here that for all $x,y\in H$, the map $P_{x,y}:\operatorname{Borel}(\mathbb{R})\rightarrow \mathbb{K}$ given by $P_{x,y}(B)=\langle P(B)x,y\rangle$ is a $\mathbb{K}$-signed measure, which is positive whenever $x=y$. Furthermore, for any measurable $\phi:\mathbb{R}\rightarrow\mathbb{R},$ there is a unique self-adjoint map, denoted $\int_{\mathbb{R}}\phi dP,$ such that $\Dom(\int_{\mathbb{R}}\phi dP)=\{x\in H\;|\; \int_{\mathbb{R}}|\phi|^2dP_{x,x}<\infty \}$ and such that, for any $x\in\Dom(\int_{\mathbb{R}}\phi dP)$ and $y\in H$, we have $$\left\langle\left(\int_{\mathbb{R}}\phi dP\right) x,y\right\rangle=\int_{\mathbb{R}}\phi dP_{x,y};$$ see, for example, \cite[Chapter X, Theorem 4.7]{conway2019course}).

The following is another form of the Spectral Theorem, which one might call the \emph{spectral measure} version of the Spectral Theorem, to distinguish it from the \emph{direct integral} version proved earlier in this paper.

\begin{thm}
    Suppose that $A$ is self-adjoint. For any Borel set $B\subseteq \mathbb{R}$, there is a bounded operator $P(B):H\rightarrow H$ for which $\langle P(B)x, y \rangle= \nu^{x,y}(B)$ for all $x,y\in H$. Moreover, $P(B)$ is a projection for each Borel set $B\subseteq \mathbb{R}$ and $P$ is a projection-valued measure for which $A=\int_{\mathbb{R}}\Id_{\mathbb{R}} dP$.
\end{thm}
\begin{proof}
    Using Definition \ref{definition_nu_measure}, it is clear that given any Borel set $B\subseteq \mathbb{R}$, the map $(x,y)\rightarrow \nu^{x,y}(B)$ is a positive semi-definite inner product on $H\times H$ for which $|\nu^{x,y}(B)|\leq\|x\| \|y\|$ holds for all $x,y\in H$. As such, there is a well-defined bounded operator $P(B)$ on $H$ satisfying the condition in the statement of the theorem.  Moreover, it is easy to see that $P(B)$ is a non-negative operator on $H$.

    Furthermore, for all $x,y\in H$, by Proposition \ref{proposition_isometry_derivatives}, we have 
    \begin{align*}
        \langle P(B)x, y\rangle &= \nu^{x,y}(B)=\int_{B} \langle U(x),U(y)\rangle d\mu=\int_{\mathbb{R}}\mathbf{1}_B\cdot \langle U(x),U(y)\rangle d\mu\\
        &=\int_{\mathbb{R}}\langle T_{\mathbf{1}_B}(U(x)),U(y)\rangle d\mu=\langle T_{\mathbf{1}_B}(U(x)),U(y)\rangle_{\mu}\\
        &=\langle (U^{-1}\circ T_{\mathbf{1}_B} \circ U) x,y\rangle.
    \end{align*}
    Thus, $P(B)=U^{-1}\circ T_{\mathbf{1}_B}\circ U$ holds for any Borel set $B$. It follows that $P$ is a spectral measure on $H$ for which $P_{x,y}=\nu^{x,y}$ for all $x,y\in H$. Since both $A$ and $\int_{\mathbb{R}}\Id_{\mathbb{R}} dP$ are self-adjoint, to conclude $A=\int_{\mathbb{R}}\Id_{\mathbb{R}}dP$, it is sufficient to show that, for any $x\in\Dom(A)$, we have $\int_{\mathbb{R}}|\Id_{\mathbb{R}}|^2d\nu^{x,x}<\infty$ and $\int_{\mathbb{R}}
    \Id_{\mathbb{R}}d\nu^{x,x}=\langle Ax, x\rangle$. For any such $x\in\Dom(A),$ we have, by Theorem \ref{theorem_conclude_spectral}, that
    \begin{align*}
        \int_{\mathbb{R}}|\Id_{\mathbb{R}}|^2d\nu^{x,x}&=\int_{\mathbb{R}} |\Id_{\mathbb{R}}|^2\cdot\|U(x)\|^2d\mu=\int_{\mathbb{R}}\|U(Ax)\|^2d\mu=\|Ax\|^2<\infty,
    \end{align*}
    and that
    \begin{align*}
        \int_{\mathbb{R}}\Id_{\mathbb{R}}d\nu^{x,x}&=\int_{\mathbb{R}}\Id_{\mathbb{R}}\cdot\|U(x)\|^2d\mu=\int_{\mathbb{R}}\langle T(U(x)),U(x)\rangle d\mu\\
        &=\langle U(Ax), U(x)\rangle_{\mu}=\langle Ax,x\rangle.
    \end{align*}

    Thus $A=\int_{\mathbb{R}}\Id_{\mathbb{R}}dP,$ concluding the proof.
\end{proof}

\appendix

\section{Appendix on measurable structures}

In this appendix, we review the necessary material on measurable structures and direct integrals of Hilbert spaces needed throughout the paper.  For more details, we recommend the reader consult \cite{dixmier_1969Ladd}.

\begin{defn}\label{def:measstructuredefinition}
Suppose that $(M,\Omega)$ is a measurable space and $(H_x)_{x\in M}$ is a family of separable Hilbert spaces.  Then a \textbf{measurable structure} (relative to this data) is a subspace $S$ of $\prod_{x\in X}H_x$ satisfying the following conditions:
\begin{enumerate}
    \item For all $Y\in \prod_{x\in X}H_x$, one has $Y\in S$ if and only if $\langle X,Y\rangle$ is measurable for all $X\in S$.
    \item There is a countable subset $S'$ of $S$ such that $\overline{\Span(\{X(x) \ : \ X\in S'\})}=H_x$ for all $x\in M$.
\end{enumerate}
The elements of $S$ are referred to as the \textbf{measurable sections} of the measurable structure.
\end{defn}
\begin{rmk}
    By the first property in the previous definition, we have that if $X\in S$ and $a:M\rightarrow\mathbb{K}$ is any measurable function, then $a\cdot X\in S.$
\end{rmk}

\begin{example}\label{example:standardmeas}
Suppose that $H_x=H$ for all $x\in M$ (so that $\prod_{x\in M}H_x=H^M$), where $H$ is some separable Hilbert space.  Then the collection of all measurable functions $M\to H$ is a measurable structure, which we refer to as the \textbf{canonical measurable structure}.  Here, and in what follows, we always view a Hilbert space as a measuable space by equipping it with its Borel $\sigma$-algebra.
\end{example}
In connection with the previous example, we mention the following fact, which is probably well-known but for which we were unable to find a reference:

\begin{lemma}\label{lem:HSmeasurablefunction}
Suppose that $(M,\Omega)$ is a measurable space and that $H$ is a separable Hilbert space.  Then a function $X:M\to H$ is measurable if and only if, for every $y\in H$, one has that the function $x\mapsto \langle X(x),y\rangle:M\to \mathbb{K}$ is measurable.    
\end{lemma}

\begin{proof}
    On the one hand, since $u\mapsto \langle u,y\rangle$ is a continuous function on $H$, it is clear that if $X$ is measurable, then $x\mapsto \langle X(x), y\rangle$ is measurable as well.  
    
    On the other hand, since $H$ is a separable metric space, any open set is the countable union of open balls. Thus, the collection of open balls generates the Borel $\sigma$-algebra on $H$; consequently, $X$ is measurable if and only if, for any $u\in H$ and $r\in\mathbb{R}_{>0}$, $X^{-1}(B_{r}(u))$ is measurable, which happens if and only if, for any $u\in H$, $x\mapsto\|X(x)-u\|^2$ is measurable. If $\{g_j\}_{j\in\mathbb{N}}$ is an orthonormal basis of $H$, Parseval's identity yields that $$\|X(x)-u\|^2=\sum_{j\in\mathbb{N}}\left|\langle X(x),g_j\rangle -\langle u, g_j \rangle \right|^2$$ for any $x\in M.$ Thus, if $x\mapsto\langle X(x), y\rangle$ is measurable for each $y\in H$, it easily follows that $x\mapsto\|X(x)-u\|^2$ is measurable.
\end{proof}

In the paper, we will need the following construction, which is probably well-known but for which we could not find a suitable reference.

\begin{defn}\label{def:meassubfamily}
Suppose that $(V_j:M\to H)_{j\in \mathbb{N}}$ is a collection of measurable functions.  For each $x\in M$ and $n\in\mathbb{N}$, we define $E_n(x)=\Span\{V_j(x)\}_{j=1}^n$. Finally, given $x\in M$, we define $$H_x=\overline{\Span\{V_j(x)\}_{j\in\mathbb{N}}}=\overline{\bigcup_{n\in\mathbb{N}}E_n(x)}\leq H.$$
\end{defn}

We wish to have the canonical measurable structure on $H^M$ induce a measurable structure on $\prod_{x\in X}H_x$ in the obvious way.  The key to doing so is the following lemma:

\begin{lemma}\label{lemma_measurable_projection}
For any $n\in\mathbb{N}$ and measurable function $Y:M\rightarrow H$, there exists measurable functions $a_1,\;\dots,\;a_n:M\rightarrow\mathbb{K}$ such that, for any $x\in M,$ we have $$\Proj_{E_n(x)}Y(x)=\sum_{j=1}^{n} a_j(x)V_j(x).$$ In particular, $x\rightarrow\Proj_{E_n(x)}Y(x)$ is measurable. Finally, the map $x\rightarrow \Proj_{H_x}Y(x)$ is measurable as well.
\end{lemma}
\begin{proof}
	 We prove the lemma by induction on $n$. If $n=1,$ then the formula is satisfied by the measurable coefficient $$a_1(x)=\mathbf{1}_{V_1^{-1}(H\setminus\{0\})}(x)\frac{\langle Y(x),V_1(x)\rangle}{\|V_1(x)\|^2}.
	$$ Thus, the lemma holds for $n=1$.
	 
	If, given $n\geq 2$ the lemma holds for $n-1,$ we show it holds for $n$.
	
	Let $\hat{V}(x)=V_n(x)-\Proj_{E_{n-1}(x)}V_n(x)$ and  $a(x)=\mathbf{1}_{\hat{V}^{-1}(H\setminus\{0\})}(x)\cdot\frac{\langle Y(x),\hat{V}(x)\rangle}{\|\hat{V}(x)\|^2},$ where $x\in M.$  Then, for any such $x$, we have that
	\begin{align*}\Proj_{E_n(x)}Y(x)&=\Proj_{E_{n-1}(x)}(Y(x))+a(x)\hat{V}(x)\\
		&=a(x)V_n(x)+\Proj_{E_{n-1}(x)}\left(Y(x)-a(x)V_n(x)\right).
	\end{align*} By hypothesis, we have that $\hat{V}:M\rightarrow H$ is measurable, from which we have that $a:M\rightarrow\mathbb{K}$ is measurable as well. Thus, $Y-a\cdot V_n$ is measurable. By hypothesis, let $a_1,\;\dots,\;a_{n-1}:M\rightarrow\mathbb{K}$ be measurable coefficients satisfying $$\Proj_{E_{n-1}(x)}\left(Y(x)-a(x)V_n(x)\right)=\sum_{j=1}^{n-1}a_j(x)V_j(x).$$ Letting $a_n=a$, we obtain $$\Proj_{E_n(x)}Y(x)=\sum_{j=1}^{n} a_j(x)V_j(x)$$ for any $x\in M$. Thus, the lemma holds for $n$, completing the induction. The fact that $x\rightarrow \Proj_{H_x}Y(x)$ is measurable follows from Lemma \ref{lem:HSmeasurablefunction} and the fact that $\langle \Proj_{H_x}Y(x), y\rangle=\lim_{n\rightarrow\infty}\langle \Proj_{E_n(x)}Y(x), y\rangle$ for any $x\in M.$
\end{proof}

\begin{prop}\label{prop:inducedmeasfamily}
  Define $S\subseteq \prod_{x\in M}H_x$ to be the collection of those functions $X\in \prod_{x\in M} H_x$ such that $X$ is measurable as an element of $H^M$.  Then $S$ is a measurable structure.
\end{prop}

\begin{proof}
We check the two conditions from Definition \ref{def:measstructuredefinition}.  The forward direction of (1) is immediate.  To check the reverse direction, suppose that $Y\in \prod_{x\in M}H_x$ is such that the map $x\mapsto \langle Y(x),X(x)\rangle:M\to \mathbb{K}$ is measurable for all $X\in S$.  We wish to show that $Y\in S$.
In other words, we want to show that, for any $y\in\ H$, $x\mapsto \langle Y(x),y\rangle$ is measurable. For any $x\in M$, we have that $\langle Y(x),y\rangle=\langle Y(x),\Proj_{H_x}y\rangle.$ Since constant maps are measurable, it follows from Lemma \ref{lem:HSmeasurablefunction} that $x\rightarrow \Proj_{H_x}y:M\to H$ is measurable, and as such it belongs to $S$. Thus, by assumption, we have that $x\rightarrow\langle Y(x),y\rangle=\langle Y(x),\Proj_{H_x}y\rangle $ is measurable.

Item (2) of Definition \ref{def:measstructuredefinition} is immediate with $S'=\{V_j\}_{j\in\mathbb{N}}$.
\end{proof}

Finally, we recall the definition of a direct integral of Hilbert spaces.  Suppose that $\mu$ is a $\sigma$-additive measure on $(M,\Omega)$ and let $S\subseteq \prod_{x\in M}H_x$ be a measurable structure.  We define $\int^{\oplus}_M H_xd\mu(x)$ to be the Hilbert space consisting of those measurable sections $X\in S$ such that $$\int_M \|X(x)\|^2d\mu(x)<\infty,$$ quotiented by those sections that are almost-everywhere $0$, and where the inner product is given by $$\langle X,Y\rangle_{\mu}=\int_M \langle X,Y\rangle d\mu.$$ That $\int_M^{\oplus}H_xd\mu$ is indeed a Hilbert space follows from \cite[Chapter II]{dixmier_1969Ladd}; the proof of this fact is very similar to the usual proof that $L^p$-spaces are complete. 


We note that in this paper, we apply the direct integral construction in two special cases:  (1) $M$ is a finite probability space with $\mu(x)>0$ for every $x\in M$, in which case $\int_M^{\oplus}H_xd\mu(x)=\prod_{x\in M}H_x$ (in the text we really work with the transferred version of the notion), and (2) $M=\mathbb{R}$ with its Borel algebra.  

\printbibliography
\end{document}